\documentclass[a4paper]{amsart}                

\usepackage[latin1]{inputenc}                  
\usepackage[T1]{fontenc}                       
\usepackage[spanish,english]{babel}            
\usepackage{graphicx}                  
\usepackage{amsmath,amssymb,amsthm}            
\usepackage{latexsym}                          
\usepackage{delarray}                          
\usepackage{bbm}                               
\usepackage{hyperref}                          
\usepackage[pdftex,usenames,dvipsnames]{color} 
\usepackage{calrsfs,eufrak,mathrsfs,upgreek}   
\usepackage{bbding,trfsigns}                   
\usepackage{wasysym}                           
\usepackage{datetime}                          
\usepackage{color}                             

\DeclareMathAlphabet{\mathpzc}{OT1}{pzc}{m}{it} 

\newtheorem{teo}{Theorem}
\newtheorem{lema}{Lemma}[section]

\newtheorem{ThA}{Theorem}

\newcommand{\R}{\mathbb{R}}

\newcommand{\Rn}{{\mathbb{R}^n}}

\DeclareMathAlphabet{\mathpzc}{OT1}{pzc}{m}{it}

\DeclareFontFamily{U}{mathx}{\hyphenchar\font45}
\DeclareFontShape{U}{mathx}{m}{n}{
      <5> <6> <7> <8> <9> <10>
      <10.95> <12> <14.4> <17.28> <20.74> <24.88>
      mathx10
      }{}
\DeclareSymbolFont{mathx}{U}{mathx}{m}{n}
\DeclareFontSubstitution{U}{mathx}{m}{n}
\DeclareMathAccent{\widecheck}{0}{mathx}{"71}
\DeclareMathAccent{\wideparen}{0}{mathx}{"75}

\allowdisplaybreaks

\author[J.J. Betancor]{Jorge J. Betancor}
\author[A.J. Castro]{Alejandro J. Castro}
\author[J.C. Fari\~na]{Juan C. Fari\~na}
\author[L. Rodr\'iguez-Mesa]{L. Rodr\'iguez-Mesa}

\address{\newline
        Jorge J. Betancor, Alejandro J. Castro, Juan C. Fari\~na and Lourdes Rodr\'iguez-Mesa \newline
        Departamento de An\'alisis Matem\'atico,
        Universidad de La Laguna, \newline
        Campus de Anchieta, Avda. Astrof\'{\i}sico Francisco S\'anchez, s/n, \newline
        38271, La Laguna (Sta. Cruz de Tenerife), Spain}
\email{jbetanco@ull.es, ajcastro@ull.es, jcfarina@ull.es, lrguez@ull.es}

\thanks{This paper was partially supported by MTM2010/17974. The second author was also supported by an FPU grant from the Goverment of Spain.}

\date{\today}

\begin{document}

\title[Solutions of Weinstein equations and Bessel Poisson integrals of BMO functions]
{Solutions of Weinstein equations representable by Bessel Poisson integrals of BMO functions}

\subjclass[2010]{42B37, 42B35. 35J15}

\keywords{Weinstein equation, Bessel Poisson integral, BMO, Carleson measure.}

\begin{abstract}
    We consider the Weinstein type equation $\mathcal{L}_\lambda u=0$ on $(0,\infty)\times (0,\infty)$, where $\mathcal{L}_\lambda =\partial ^2_t+\partial^2 _x-\frac{\lambda(\lambda-1)}{x^2}$, with $\lambda >1$. In this paper we characterize the solutions of $\mathcal{L}_\lambda u=0$ on $(0,\infty)\times(0,\infty)$ representable by Bessel-Poisson integrals of $BMO$-functions as those ones satisfying certain Carleson properties.
\end{abstract}

\maketitle

\section{Introduction}\label{sec:intro}

The space $BMO(\Rn)$ of bounded mean oscillation functions in $\Rn$ was introduced by John and Nirenberg (\cite{JN}) in the context of partial differential equations. A function $f\in L^1_{\rm loc}(\Rn)$ is in $BMO(\Rn)$ provided that
$$
\|f\|_{BMO(\Rn)}:=\sup_B\frac{1}{|B|}\int_B |f(x)-f_B|dx < \infty,
$$
where the supremum is taken over all balls $B$ in $\Rn$. Here, $|B|$ denotes the Lebesgue measure of $B$ and $f_B$ represents the average of $f$ on $B$, that is, $f_B=\frac{1}{|B|}\int_Bf(x)dx$. By identifying those functions that differ by a constant, $(BMO(\Rn),\|\cdot\|_{BMO(\Rn)})$ is a Banach space.

A celebrated result of Fefferman and Stein (\cite{FS}) establishes that $BMO(\Rn)$ is the dual space of the Hardy space $H^1(\Rn)$. The spaces $H^1(\Rn)$ and $BMO(\Rn)$ turned out to be the correct substitutes for $L^1(\Rn)$ and $L^\infty(\Rn)$, respectively, as the domain and the target spaces of operators appearing in harmonic analysis.

Since Fefferman and Stein's paper (\cite{FS}) appeared, the space of bounded mean oscillation functions has motivated the investigations of many mathematicians (see, for instance, \cite{ChF}, \cite{CLMS}, \cite{CS}, \cite{DY2}, \cite{DY3}, \cite{FL}, \cite{HM1},  \cite{HW}, \cite{LT}, \cite{MMNO}, \cite{NTV}, \cite{PS}, \cite{RS}, \cite{T} and \cite{Va}).

The space $BMO(\Rn)$ is closely connected to certain positive measures in $\mathbb R^{n+1}_+$ known as Carleson measures. These measures were introduced by Carleson to solve the corona problem (\cite{Ca}). A positive measure $\mu$ on $\mathbb R^{n+1}_+$ is called a Carleson measure when
\begin{equation}\label{1.2}
\|\mu\|_{\mathcal{C}}:=\sup_Q\frac{\mu(Q\times (0,\ell(Q))}{|Q|} < \infty,
\end{equation}
where the supremum is taken over all cubes $Q$ in $\mathbb R^n$. Here $\ell(Q)$ denotes the length of the edge of $Q$. 

If $f$ is a measurable function on $\Rn$ such that $\int_{\Rn}|f(x)|(1+|x|)^{-n-1}dx <\infty$, then, for every $t>0$, the Poisson integral $P_t(f)$ of $f$ is defined by
$$
P_t(f)(x)=\int_{\mathbb R^n} P_t(x-y)f(y)dy, \;\; x \in \Rn\mbox{ and }t>0,
$$
where
$$
P_t(z) = \frac{\Gamma(n+1/2)}{\pi^{n+1/2}}\frac{t}{(|z|^2+t^2)^{(n+1)/2}},\;\;z\in \Rn\;\;\mbox{and}\; t>0.
$$
The characterization of the bounded mean oscillation functions via Carleson measures was given by Fefferman and Stein.
\begin{ThA}\label{th1.1}{\rm (}\cite[p. 145]{FS}{\rm )} Let $f \in L^1_{\rm loc}(\Rn)$. Then, $f \in BMO(\Rn)$ if, and only if, $\int_\Rn|f(x)|(1+|x|)^{-n-1}dx <\infty$ and the measure $t|\nabla P_t(f)(x)|^2dxdt$ is Carleson in $\mathbb R^{n+1}_+$, where $\nabla=\left(\partial _{x_1},\ldots, \partial_{x_n},\partial_{t}\right)$.
\end{ThA}
Some versions of this result for $BMO$-type spaces associated with operators have been established in the last decade (see \cite{BCFR1}, \cite{DYZ}, \cite{DGMTZ}, \cite{HM1} and  \cite{LTZ}, amongst others).

Theorem \ref{th1.1} was completed by Fabes, Johnson and Neri (\cite{FJN2} and \cite{FJN1}). An harmonic function $u$ defined on $\mathbb R^{n+1}_+$ is said to be in $HMO(\mathbb R^{n+1}_+)$ provided that
$$
\sup_Q\frac{1}{|Q|}\int_0^{\ell(Q)}\int_Q|t\nabla u(x,t)|^2\frac{dxdt}{t}<\infty,
$$
where the supremun is taken over all cubes in $\Rn$. Theorem \ref{th1.1} implies that $P_t(BMO(\Rn))\subseteq HMO(\mathbb R^{n+1}_+)$ suitably understood. The equality is established in the following.
\begin{ThA}\label{th1.2}{\rm (}\cite[Theorem 1.0]{FJN1}{\rm )}.
A function $u \in HMO(\mathbb R^{n+1}_+)$ if, and only if, there exits $f \in BMO(\Rn)$ such that $u(x,t)=P_t(f)(x)$, $(x,t)\in \mathbb R^{n+1}_+$.
\end{ThA}
Other proof of Theorem \ref{th1.2} can be encountered in \cite[Theorem 1.0]{FJN2}. The result in Theorem \ref{th1.2} was also proved by Fabes and Neri \cite{FN} when $\mathbb R^{n+1}_+$ is replaced by a bounded starlike Lipschitz domain (\cite[Theorem p. 35]{FN}). More recently, Chen  \cite[Theorems A and 18]{Jie} and Chen and Luo \cite{CL} characterized the harmonic functions in $M\times (0,\infty)$, where $M$ is a complete connected Riemannian manifold with positive Ricci curvature, having $BMO(M)$ traces. In \cite[Theorem 1.1]{DYZ} Duong, Yang, and Zhang have extended Theorem \ref{th1.2} to Poisson integrals and $BMO$ functions in the Schr\"odinger operator setting.

Our objective in this paper is to establish a version of Theorem \ref{th1.2} in the Bessel operator context.

The study of harmonic analysis associated with Bessel operators was began in a systematic way by Muckenhoupt and Stein (\cite{MS}). In the last decade Bessel harmonic analysis has been developed (see, for instance, \cite{BBFMT},  \cite{BCC2}, \cite{BFMR}, \cite{BFS}, \cite{BMR}, \cite{Dz}, \cite{NoSt} and \cite{Vi}).

For $\lambda >0$ we consider the Bessel operator on $(0,\infty)$
$$
B_\lambda=\frac{d^2}{dx^2}-\frac{\lambda(\lambda -1)}{x^2}= x^{-\lambda }\frac{d}{dx}x^{2\lambda }\frac{d}{dx}x^{-\lambda}.
$$
According to \cite[\S 16]{MS} the Poisson semigroup $\{P_t^\lambda\}_{t>0}$ associated with the Bessel operator  $B_\lambda$ is defined as follows. If $f\in L^p(0,\infty)$, $1\leq p \leq \infty$,
$$
P_t^\lambda(f)(x)=\int_0^\infty P_t^\lambda(x,y) f(y) dy,\;\; x,t\in (0,\infty),
$$
where
$$
P_t^\lambda(x,y)=\frac{2\lambda(xy)^\lambda t}{\pi}\int_0^\pi \frac{(\sin\theta)^{2\lambda-1}}{((x-y)^2+t^2+2xy(1-\cos\theta))^{\lambda +1}}d\theta,\;\;x,y,t\in (0,\infty).
$$
$L^p$-boundedness properties of $\{P_t^\lambda\}_{t>0}$ and the associated maximal operators were studied in \cite{BS2} and \cite{NoSt3}. The semigroup $\{P_t^\lambda\}_{t>0}$ has not the Markovian property, that is, $P_t^\lambda$ does not map constants into constants. This fact produces technical difficulties when studying Bessel Poisson semigroups on functions of bounded mean oscillation (see \cite{BCFR1}).

We denote  by $BMO_{\rm o}(\R)$ the space of all those odd functions $f \in BMO(\R)$. Let $1<p<\infty$.  For every $f \in BMO_{\rm o}(\R)$ there exists $C>0$ such that,
\begin{equation}\label{1.3}
\left\{\frac{1}{|I|}\int_I|f(x)-f_I|^p dx\right\}^{1/p} \leq C,
\end{equation}
for every interval $I=(a,b)$, $0<a<b<\infty$ and
\begin{equation}\label{1.4}
\left\{\frac{1}{|I|}\int_I|f(x)|^p dx\right\}^{1/p} \leq C,
\end{equation}
for each interval $I=(0,b)$, $0<b<\infty$. Moreover, the quantity $\inf\{C>0: (\ref{1.3})\;\mbox{and}\;(\ref{1.4})\;\mbox{hold}\}$ is equivalent to $\|f\|_{BMO(\R)}$.

Conversely, if $f\in L^1_{\rm loc}[0,\infty)$ satisfies (\ref{1.3}) and (\ref{1.4}) for all admisible intervals, then the odd extension $f_{\rm o}$ of $f$ to $\R$ belongs to $BMO_{\rm o}(\R)$ and $\|f_{\rm o}\|_{BMO(\R)}$ is equivalent to the quantity $\|f\|_{BMO_{\rm o}(\mathbb{R})}:=\inf\{C>0: (\ref{1.3})\;\mbox{and}\;(\ref{1.4})\;\mbox{hold}\}$ (see \cite[Proposition 12]{Ger} and \cite[Corollary p. 144]{Ste2}). In this case we also say that $f\in BMO_{\rm o}(\mathbb{R})$.

As in (\ref{1.2}) we say that a positive measure $\mu$ on $(0,\infty)\times (0,\infty)$ is Carleson when
$$
\|\mu\|_{\mathcal{C}}:=\sup_I\frac{\mu(I\times(0,|I|))}{|I|}<\infty,
$$
where the supremun is taken over all bounded intervals $I \subset (0,\infty)$.

In \cite{BCFR1} it was proved a Bessel version of Theorem \ref{th1.1}.
\begin{ThA}\label{th1.3} {\rm (}\cite[Theorem 1.1]{BCFR1}{\rm )}. Let $\lambda >0$. Assume that $f \in L^1_{\rm loc}[0,\infty)$. Then, the following assertions are equivalent.
\begin{enumerate}
\item[$(i)$] $f \in BMO_{\rm o}(\R)$.
\item[$(ii)$] $(1+x^2)^{-1}f \in L^1(0,\infty)$ and
$$
d\gamma_f(x,t)=\left|t \partial _tP_t^\lambda(f)(x)\right|^2 \frac{dxdt}{t}
$$
is a Carleson measure on $(0,\infty)\times(0,\infty)$.
\end{enumerate}
Moreover, the quantities $\|f\|^2_{BMO_{\rm o}(\R)}$ and $\|\gamma_f\|_{\mathcal{C}}$ are equivalent.
\end{ThA}

\noindent{\sc Remark}. Another characterization of $BMO_{\rm o}(\R)$, slightly different to $(ii)$ in Theorem \ref{th1.3} and that will be used in Section \ref{sec: (ii)(i)}, is given in Lemma \ref{equiv}.
\vspace{0.2cm}

If $\Omega \subseteq (0,\infty)\times \R$ we say that a function $u\in C^2(\Omega)$ is $\lambda$-harmonic provided that 
$$
\partial_t^2u(x,t)+B_{\lambda,x}u(x,t)=0, \quad (x,t) \in \Omega .
$$
 The operator $\partial_t^2+ B_{\lambda,x}$ is related to the Weinstein operator associated with the generalized  axially symmetric potential theory (see \cite{CR} and the references there). We can write $B_\lambda=-D_\lambda^*D_\lambda$, where $D_\lambda=x^\lambda\frac{d}{dx}x^{-\lambda}$ and $D^*_\lambda$ is the formal adjoint operator of $D_\lambda$ in $L^2(0,\infty)$. We define the $\lambda$-gradient $\nabla_\lambda$ by
$$
\nabla_\lambda=\left(D_{\lambda,x}, \partial _t\right).
$$
The main result of this paper is the following.
\begin{teo}\label{th1.4}
Let $\lambda >1$. Assume that $u$ is a $\lambda$-harmonic function on $(0,\infty)\times(0,\infty)$ such that $x^{-\lambda}u(x,t) \in C^{\infty}(\R \times (0,\infty))$ and is even in the $x$-variable. Then, the following assertions are equivalent.
\begin{enumerate}
\item[$(i)$] There exists $f \in BMO_{\rm o}(\R)$ such that $u(x,t)= P_t^\lambda(f)(x)$, $(x,t)\in (0,\infty)\times(0,\infty)$.
\item[$(ii)$] The measure
$$
d\mu_\lambda(x,t)=|t\nabla_\lambda u(x,t)|^2\frac{dxdt}{t}
$$
is Carleson on $(0,\infty)\times(0,\infty)$.
\end{enumerate}
Moreover, the quantities $\|f\|_{BMO_{\rm o}(\R)}^2$ and $\|\mu_\lambda \|_{\mathcal{C}}$ are equivalent.
\end{teo}

Note that the  property $(ii)$ in Theorem \ref{th1.4} is stronger than the condition $(ii)$ in Theorem \ref{th1.3}.

In the next sections we prove Theorem \ref{th1.4}. In the sequel by $C$ we always denote a positive constant not necessarily the same in each occurrence.

\section{Proof of $(i)\Longrightarrow (ii)$ in Theorem \ref{th1.4}}\label{sec: (i)(ii)}
As it can be observed along the proof, this part of Theorem \ref{th1.4} is valid for $\lambda >0$.

Assume that $u(x,t)=P_t^ \lambda (f)(x)$, $x,t\in (0,\infty )$, for a certain $f\in BMO_{\rm o}(\mathbb{R})$. According to Theorem \ref{th1.3} the measure
$$
d\gamma _f(x,t)=\left|t\partial_tP_t^\lambda (f)(x)\right|^2\frac{dxdt}{t}
$$
is Carleson on $(0,\infty)\times (0,\infty )$. Moreover, we have that
$$
\|\gamma _f\|_\mathcal{C}\leq C\|f\|^2_{BMO_{\rm o}(\mathbb{R})},
$$
where $C>0$ does not depend on $f$.

We are going to see that the measure
$$
d\rho _f(x,t)=|tD_{\lambda ,x}P_t^\lambda (f)(x)|^2\frac{dxdt}{t}
$$
is Carleson on $(0,\infty)\times (0,\infty )$ and that
$$
\|\rho _f\|_\mathcal{C}\leq C\|f\|^2_{BMO_{\rm o}(\mathbb{R})},
$$
for certain $C>0$ which does not depend on $f$.

Let $I=(a,b)$ where $0\leq a<b<\infty$. We decompose $f$ as follows
$$
f=(f-f_{2I})\chi _{2I}+(f-f_{2I})\chi _{(0,\infty )\setminus 2I}+f_{2I}=:f_1+f_2+f_3.
$$
Here $2I=(x_I-|I|, x_I+|I|)\cap (0,\infty )$ and $x_I=(a+b)/2$.

We will prove that
\begin{equation}\label{2.1}
\frac{1}{|I|}\int_0^{|I|}\int_I|tD_{\lambda ,x}P_t^\lambda (f_j)(x)|^2\frac{dxdt}{t}\leq C\|f\|_{BMO_{\rm o}(\mathbb{R})}^2,\quad j=1,2,3,
\end{equation}
for certain $C>0$ independent of $I$ and $f$.

\subsection{Proof of (\ref{2.1}) for $j=1$}\label{subs2.1} We introduce the Littlewood-Paley function $g_\lambda $ defined by
$$
g_\lambda (F)(x)=\left(\int_0^\infty |tD_{\lambda ,x}P_t^\lambda (F)(x)|^2\frac{dt}{t}\right)^{1/2},\quad x\in (0,\infty ),
$$
for every $F\in L^2(0,\infty )$.
\begin{lema}\label{Lem2.1}
Let $\lambda >0$. The Littlewood-Paley function $g_\lambda$ is a bounded (sublinear) operator from $L^2(0,\infty )$ into itself.
\end{lema}
\begin{proof}
We consider the Hankel transformation $h_\lambda$ defined by
$$
h_\lambda (F)(x)=\int_0^\infty \sqrt{xy}J_{\lambda -1/2}(xy)F(y)dy,\quad x\in (0,\infty ),
$$
for every $F\in L^1(0,\infty )$. Here $J_\nu$ denotes the Bessel function of the first kind and order $\nu$. The transformation $h_\lambda$ can be extended from $L^1(0,\infty )\cap L^2(0,\infty)$ to $L^2(0,\infty )$ as an isometry of $L^2(0,\infty )$, where $h_\lambda ^{-1}=h_\lambda $ (\cite[Ch. VIII]{Ti}).

Let $F\in L^2(0,\infty )$. According to \cite[(16.1')]{MS} we have  that
$$
P_t^\lambda (F)(x)=h_\lambda (e^{-yt}h_\lambda (F))(x),\quad x,t\in (0,\infty ).
$$
Since $\frac{d}{dz}(z^{-\nu }J_\nu (z))=-z^{-\nu }J_{\nu +1}(z)$, $z\in (0,\infty )$, and $e^{-yt}h_\lambda (F)\in L^1(0,\infty )$, $t>0$, we get
$$
D_{\lambda ,x}P_t^\lambda (F)(x)=-h_{\lambda +1}(ye^{-yt}h_\lambda (F))(x),\quad x,t\in (0,\infty ).
$$
Then,
\begin{align*}
\|g_\lambda (F)\|_{L^2(0,\infty )}^2&=\int_0^\infty \int_0^\infty |tD_{\lambda ,x}P_t^\lambda (F)(x)|^2\frac{dxdt}{t}\\
&=\int_0^\infty t\int_0^\infty |h_{\lambda +1}(ye^{-ty}h_\lambda (F)(y))(x)|^2dxdt\\
&=\int_0^\infty t\int_0^\infty y^2e^{-2ty}|h_\lambda (F)(y)|^2dydt\\
&=\int_0^\infty y^2|h_\lambda (F)(y)|^2\int_0^\infty te^{-2ty}dtdy=\frac{1}{4}\|h_\lambda (F)\|_{L^2(0,\infty )}^2=\frac{1}{4}\|F\|_{L^2(0,\infty )}^2.
\end{align*}
\end{proof}

Lemma \ref{Lem2.1} leads to
\begin{align*}
\frac{1}{|I|}\int_0^{|I|}\int_I|tD_{\lambda ,x}P_t^\lambda (f_1)(x)|^2\frac{dxdt}{t}&\leq \frac{1}{|I|}\int_0^\infty |g_\lambda (f_1)(x)|^2dx\leq \frac{C}{|I|}\int_0^\infty |f_1(y)|^2dy\nonumber\\
&=\frac{C}{|I|}\int_{2I} |f(y)-f_{2I}|^2dy\leq C\|f\|_{BMO_{\rm o}(\mathbb{R})}^2,
\end{align*}
being $C$ independent of $I$ and $f$.
\subsection{Proof of (\ref{2.1}) for $j=2$}\label{subs2.2}
First of all we establish the following estimation for the kernel $D_{\lambda ,x}P_t^\lambda (x,y)$, $x,y,t\in (0,\infty )$.

\begin{lema}\label{Lem2.2}
Let $\lambda >0$. Then,
$$
|D_{\lambda ,x}P_t^\lambda (x,y)|\leq \frac{C}{(x-y)^2+t^2},\quad x,y,t\in (0,\infty ).
$$
\end{lema}

\begin{proof}
We write the following decomposition
\begin{align}\label{2.*}
D_{\lambda ,x}P_t^\lambda (x,y)&=-\frac{4\lambda (\lambda +1)}{\pi }(xy)^\lambda t\int_0^\pi \frac{(\sin \theta )^{2\lambda -1}(x-y\cos \theta)}{((x-y)^2+t^2+2xy(1-\cos \theta ))^{\lambda +2}}d\theta \nonumber\\
&=:P_t^{\lambda ,1}(x,y)+P_t^{\lambda ,2}(x,y),\quad x,y, t\in (0,\infty ),
\end{align}
where
$$
P_t^{\lambda ,1}(x,y)=-\frac{4\lambda (\lambda +1)}{\pi }(xy)^\lambda t\int_0^{\pi/2} \frac{(\sin \theta )^{2\lambda -1}(x-y\cos \theta)}{((x-y)^2+t^2+2xy(1-\cos \theta ))^{\lambda +2}}d\theta
,\quad x,y,t\in (0,\infty ).
$$
We have that
\begin{equation}\label{2.3}
|P_t^{\lambda ,2}(x,y)|\leq C(xy)^\lambda t\frac{x+y}{((x-y)^2+t^2+2xy)^{\lambda +2}}\leq \frac{C}{(x-y)^2+t^2},\quad x,y,t\in (0,\infty ).
\end{equation}

On the other hand, since 
$$
|x-y\cos \theta |\leq |x-y|+\min \{x,y\}(1-\cos \theta ), \quad x,y\in (0,\infty ),\; \theta \in \mathbb{R},
$$
and $\sin \theta \sim \theta$ and $2(1-\cos \theta )\sim \theta ^2$, $\theta \in [0,\pi /2]$, it follows that
$$
|P_t^{\lambda ,1}(x,y)|\leq C(P_t^{\lambda ,1,1}(x,y)+P_t^{\lambda ,1,2}(x,y)),\quad x,y,t\in (0,\infty),
$$
where
$$
P_t^{\lambda ,1,1}(x,y)=(xy)^\lambda t|x-y|\int_0^{\pi/2} \frac{\theta ^{2\lambda -1}}{((x-y)^2+t^2+xy\theta ^2)^{\lambda +2}}d\theta
,\quad x,y,t\in (0,\infty ),
$$
and
$$
P_t^{\lambda ,1,2}(x,y)=(xy)^\lambda t\min\{x,y\}\int_0^{\pi/2} \frac{\theta ^{2\lambda +1}}{((x-y)^2+t^2+xy\theta ^2)^{\lambda +2}}d\theta
,\quad x,y,t\in (0,\infty ).
$$
We get
\begin{align*}
P_t^{\lambda ,1,1}(x,y)&\leq C(xy)^\lambda t|x-y|\int_0^\infty \frac{\theta ^{2\lambda -1}}{((x-y)^2+t^2+xy\theta ^2)^{\lambda +2}}d\theta\\
&\leq C\frac{ t|x-y|}{((x-y)^2+t^2)^2}\leq \frac{C}{(x-y)^2+t^2},\quad x,y,t\in (0,\infty ),
\end{align*}
and
\begin{align*}
P_t^{\lambda ,1,2}(x,y)&\leq C\frac{(xy)^\lambda t\min \{x,y\}}{((x-y)^2+t^2)^{3/2}}\int_0^{\pi/2}\frac{\theta ^{2\lambda +1}}{(xy\theta ^2)^{\lambda +1/2}}d\theta\\
&\leq \frac{C}{(x-y)^2+t^2}\frac{\min\{x,y\}}{\sqrt{xy}}\leq \frac{C}{(x-y)^2+t^2},\quad x,y,t\in (0,\infty ).
\end{align*}
Hence,
\begin{equation}\label{2.4}
|P_t^{\lambda ,1}(x,y)|\leq \frac{C}{(x-y)^2+t^2},\quad x,y,t\in (0,\infty ).
\end{equation}

Thus, the result follows from (\ref{2.*}), (\ref{2.3}) and (\ref{2.4}).
\end{proof}

We now proceed as in \cite[pp. 468-469]{BCFR1}. By Lemma \ref{Lem2.2} we can write
\begin{align*}
|D_{\lambda ,x}P_t^\lambda (f_2)(x)|&\leq C\int_{(0,\infty )\setminus 2I}\frac{|f(y)-f_{2I}|}{(x-y)^2+t^2}dy\leq C\int_{(0,\infty )\setminus 2I}\frac{|f(y)-f_{2I}|}{(x_I-y)^2+t^2}dy\\
&\leq \frac{C}{|I|}\sum_{k=1}^\infty \frac{1}{2^k}\left(\frac{1}{2^k|I|}\int_{2^{k+1}I}|f(y)-f_{2^{k+1}I}|dy+|f_{2^{k+1}I}-f_{2I}|\right)\\
&\leq \frac{C}{|I|}\|f\|_{BMO_{\rm o}(\mathbb{R})},\quad x\in I\mbox{ and  }t>0.
\end{align*}
In the last inequality we have taken into account \cite[Ch. VI (1.3)]{G} and that, if $k\in \mathbb{N}\setminus\{0\}$ and $2^k|I|>x_I$, then $2^{k+1}I\subset (0,2^{k+1}|I|)$ and
$$
\int_{2^{k+1}I}|f(y)-f_{2^{k+1}I}|dy\leq \int_0^{2^{k+1}|I|}(|f(y)|+|f_{2^{k+1}I}|)dy\leq 2^{k+3}|I|\|f\|_{BMO_{\rm o}(\mathbb{R})}.
$$
We conclude that
$$
\frac{1}{|I|}\int_0^{|I|}\int_I|tD_{\lambda ,x}P_t^\lambda (f_2)(x)|^2\frac{dxdt}{t}\leq C\|f\|_{BMO_{\rm o}(\mathbb{R})}^2,
$$
with $C$ independent of $I$ and $f$.
\subsection{Proof of (\ref{2.1}) for $j=3$} Note firstly that
$$
|f_{2I}|\leq \frac{1}{|I|}\int_{2I}|f(y)|dy\leq \frac{x_I+|I|}{|I|}\|f\|_{BMO_{\rm o}(\mathbb{R})}.
$$
Then, estimation (\ref{2.1}) for $j=3$ will be proved once we show the following.
\begin{lema}\label{Lem2.3} Let $\lambda >0$. There exists $C>0$ such that
\begin{equation}\label{2.6}
\frac{(x_J+|J|)^2}{|J|^3}\int_0^{|J|}\int_J|tD_{\lambda ,x}P_t^\lambda (1)(x)|^2\frac{dxdt}{t}\leq C,
\end{equation}
for every bounded interval $J$ on $(0,\infty )$.
\end{lema}
\begin{proof}
We take in mind the decomposition (\ref{2.*}). As in (\ref{2.3}) we get 
$$
|P_t^{\lambda ,2}(x,y)|\leq C\frac{(x+y)t}{(x^2+y^2+t^2)^2}\leq\frac{C}{x^2+y^2+t^2},\quad x,y,t\in (0,\infty ).
$$
Then,
\begin{equation}\label{5.1}
\left|\int_0^\infty P_t^{\lambda ,2}(x,y)dy\right|\leq C\int_0^\infty \frac{dy}{x^2+y^2+t^2}\leq \frac{C}{x+t},\quad x,t\in (0,\infty ).
\end{equation}

Now we write the following spliting
\begin{align*}
\int_0^\infty P_t^{\lambda ,1}(x,y)dy&=\left(\int_0^{x/2}+\int_{x/2}^{3x/2}+\int_{3x/2}^\infty \right)P_t^{\lambda ,1}(x,y)dy\\
&=:Q_t^1(x)+Q_t^2(x)+Q_t^3(x),\quad x\in (0,\infty ).
\end{align*}
According to (\ref{2.4}) we get
$$
|Q_t^1(x)|\leq C\int_0^{x/2}\frac{dy}{(x-y)^2+t^2}\leq C\frac{x}{x^2+t^2}\leq \frac{C}{x+t},\quad x,t\in (0,\infty ),
$$
and
$$
|Q_t^3(x)|\leq C\int_{3x/2}^\infty \frac{dy}{(y-x+t)^2}\leq \frac{C}{x+t},\quad x,t\in (0,\infty ).
$$

We decompose $Q_t^2(x)$, $t,x\in (0,\infty )$, in the following way.
\begin{align*}
Q_t^2(x)&=-\frac{4\lambda (\lambda +1)}{\pi }t\int_{x/2}^{3x/2}(xy)^\lambda y\int_0^{\pi/2}\frac{(1-\cos \theta )(\sin \theta )^{2\lambda -1}}{((x-y)^2+t^2+2xy(1-\cos \theta ))^{\lambda +2}}d\theta dy\\
&\hspace{0.5cm}-\frac{4\lambda (\lambda +1)}{\pi }t\int_{x/2}^{3x/2}(xy)^\lambda (x-y)\\
&\hspace{0.5cm}\times \int_0^{\pi/2}\left(\frac{(\sin \theta )^{2\lambda -1}}{((x-y)^2+t^2+2xy(1-\cos \theta ))^{\lambda +2}}-\frac{\theta ^{2\lambda -1}}{((x-y)^2+t^2+xy\theta^2 )^{\lambda +2}}\right)d\theta dy\\
&\hspace{0.5cm}+\frac{4\lambda (\lambda +1)}{\pi }t\int_{x/2}^{3x/2}(xy)^\lambda (x-y)\int_{\pi/2}^\infty \frac{\theta ^{2\lambda -1}}{((x-y)^2+t^2+xy\theta^2 )^{\lambda +2}}d\theta dy\\
&\hspace{0.5cm}-\frac{4\lambda (\lambda +1)}{\pi }t\int_{x/2}^{3x/2}(xy)^\lambda (x-y)\int_0^\infty \frac{\theta ^{2\lambda -1}}{((x-y)^2+t^2+xy\theta^2 )^{\lambda +2}}d\theta dy\\
&=:\sum_{j=1}^4I_j(x,t),\quad x,t\in (0,\infty ).
\end{align*}
Observe firstly that $I_4(x,t)=0$, $t,x\in (0,\infty )$. Indeed, we have that
\begin{align*}
I_4(x,t)&=-\frac{4\lambda (\lambda +1)}{\pi }t\int_0^\infty \frac{u^{2\lambda -1}}{(1+u^2)^{\lambda +2}}du\int_{x/2}^{3x/2}
\frac{(xy)^\lambda (x-y)}{(xy)^\lambda ((x-y)^2+t^2)^2 }dy\\
&=-\frac{2t}{\pi}\int_{-x/2}^{x/2}\frac{z}{(z^2+t^2)^2}dz=0,\quad x,t\in (0,\infty ).
\end{align*}
We are going to see that
$$
|I_j(x,t)|\leq \frac{C}{x+t},\quad x,t\in (0,\infty )\mbox{ and }j=1,2,3.
$$

Since $2(1-\cos \theta)\sim \theta ^2$ and $\sin \theta \sim \theta$, when $\theta \in [0,\pi /2]$,  we can write 
\begin{align*}
|I_1(x,t)|&\leq Ctx^{2\lambda +1}\int_{x/2}^{3x/2}\int_0^{\pi /2}\frac{\theta ^{2\lambda +1}}{((x-y)^2+t^2+(x\theta )^2 )^{\lambda +2}}d\theta dy\\
&\leq Ct\int_{x/2}^{3x/2}\int_0^\infty\frac{d\theta dy}{((x-y)^2+t^2+(x\theta )^2 )^{3/2}}\\
&\leq C\frac{t}{x}\int_{x/2}^{3x/2}\frac{dy}{(x-y)^2+t^2}\leq C\frac{t}{x}\int_0^{x/2}\frac{dz}{z^2+t^2}\\
&\leq C\frac{t}{x}\int_0^{x/2}\frac{dz}{(z+t)^2}\leq\frac{C}{x+t},\quad x,t\in (0,\infty ).
\end{align*}
Also,
\begin{align*}
|I_3(x,t)|&\leq Ctx^{2\lambda }\int_{x/2}^{3x/2}\int_{\pi /2}^\infty \frac{\theta ^{2\lambda -1}}{((x-y)^2+t^2+(x\theta)^2 )^{\lambda +3/2}}d\theta dy\\
&\leq Ct\int_{x/2}^{3x/2}\frac{1}{((x-y)^2+t^2)^{3/2}}\int_{\frac{\pi}{2}\frac{x}{\sqrt{(x-y)^2+t^2}}}^\infty \frac{u^{2\lambda -1}}{(1+u^2)^{\lambda+3/2}}dudy\\
&\leq C\frac{t}{x}\int_{x/2}^{3x/2}\frac{\sqrt{(x-y)^2+t^2}}{((x-y)^2+t^2)^{3/2}}dy\int_0^\infty \frac{u^{2\lambda }}{(1+u^2)^{\lambda+3/2}}du\\
&\leq C\frac{t}{x}\int_{x/2}^{3x/2}\frac{dy}{(x-y)^2+t^2}\leq \frac{C}{x+t},\quad x,t\in (0,\infty ).
\end{align*}
By using that $|(\sin \theta  )^{2\lambda -1}-\theta ^{2\lambda -1}|\leq C\theta ^{2\lambda +1}$, $\theta \in (0,\pi /2)$, and that
$$
\left|\frac{1}{((x-y)^2+t^2+2xy(1-\cos \theta ))^{\lambda +2}}-\frac{1}{((x-y)^2+t^2+xy\theta^2 )^{\lambda +2}}\right|\leq C\frac{xy\theta ^4}{((x-y)^2+t^2+xy\theta^2 )^{\lambda +3}},
$$
for each $\theta \in (0,\pi /2)$ and $t,x,y\in (0,\infty )$, we obtain
\begin{align*}
\left|\frac{(\sin \theta )^{2\lambda -1}}{((x-y)^2+t^2+2xy(1-\cos \theta ))^{\lambda +2}}-\frac{\theta ^{2\lambda -1}}{((x-y)^2+t^2+xy\theta^2 )^{\lambda +2}}\right|&\\
&\hspace{-8cm}\leq C\left(\frac{\theta ^{2\lambda +1}}{((x-y)^2+t^2+2xy(1-\cos \theta ))^{\lambda +2}}+\theta ^{2\lambda -1}\frac{xy\theta ^4}{((x-y)^2+t^2+xy\theta ^2)^{\lambda +3}}\right)\\
&\hspace{-8cm}\leq C\frac{\theta ^{2\lambda +1}}{((x-y)^2+t^2+xy\theta ^2)^{\lambda +2}}, \quad \theta \in \Big(0,\frac{\pi }{2}\Big), \;x,y,t\in (0,\infty ).
\end{align*}
Then,
\begin{align*}
|I_2(x,t)|&\leq Ctx^{2\lambda }\int_{x/2}^{3x/2}|x-y|\int_0^{\pi /2}\frac{\theta ^{2\lambda +1}}{((x-y)^2+t^2+(x\theta )^2)^{\lambda +2}}d\theta dy\\
&\leq Ctx^{2\lambda }\int_{x/2}^{3x/2}\int_0^{\pi /2}\frac{\theta ^{2\lambda +1}}{((x-y)^2+t^2+(x\theta )^2)^{\lambda +3/2}}d\theta dy\\
&\leq Ctx^{2\lambda }\int_{x/2}^{3x/2}\frac{dy}{(x-y)^2+t^2}\int_0^{\pi /2}\frac{\theta ^{2\lambda +1}}{(x\theta )^{2\lambda +1}}d\theta\\
&\leq C\frac{t}{x}\int_{x/2}^{3x/2}\frac{dy}{(x-y)^2+t^2}\leq \frac{C}{x+t},\quad x,t\in (0,\infty ).
\end{align*}
We conclude that
$$
|Q_t^2(x)|\leq \frac{C}{x+t},\quad x,t\in (0,\infty ).
$$
Hence, 
\begin{equation}\label{7.1}
\left|\int_0^\infty P_t^{\lambda ,1}(x,y)dy\right|\leq \frac{C}{x+t},\quad x,t\in (0,\infty ).
\end{equation}

Let $J$ a bounded interval in $(0,\infty )$. If $x_J<|J|$, we obtain by (\ref{5.1}) and (\ref{7.1})
\begin{align*}
\frac{(x_J+|J|)^2}{|J|^3}\int_0^{|J|}\int_Jt\left|\int_0^\infty |D_{\lambda ,x}P_t^\lambda(x,y)dy\right|^2dxdt&\leq \frac{C}{|J|}\int_0^{|J|}\int_0^{2|J|}\frac{t}{(x+t)^2}dxdt\\
&\hspace{-5.5cm}\leq \frac{C}{|J|}\int_0^{|J|}\frac{dt}{\sqrt{t}}\int_0^{2|J|}\frac{dx}{\sqrt{x}}\leq C.
\end{align*}
If $x_J\geq |J|$, again by (\ref{5.1}) and (\ref{7.1}) we can write
\begin{align*}
\frac{(x_J+|J|)^2}{|J|^3}\int_0^{|J|}\int_Jt\left|\int_0^\infty |D_{\lambda ,x}P_t^\lambda(x,y)dy\right|^2dxdt&\leq C\frac{(x_J+|J|)^2}{|J|^3}\int_0^{|J|}tdt\int_{x_J-|J|/2}^{x_J+|J|/2}\frac{dx}{x^2}\\
&\hspace{-4cm}\leq C\frac{(x_J+|J|)^2}{|J|}\left(\frac{1}{x_J-|J|/2}-\frac{1}{x_J+|J|/2}\right)\leq C\frac{(x_J+|J|)^2}{(x_J+|J|/2)(x_J-|J|/2)}\\
&\hspace{-4cm}\leq C\frac{x_J+|J|}{2x_J-|J|}\leq C.
\end{align*}
Note that the constant $C$ does not depend on $J$. Thus, (\ref{2.6}) is established.
\end{proof}
By considering Lemma \ref{Lem2.3} and the estimate for $|f_{2I}|$ we deduce that
$$
\frac{1}{|I|}\int_0^{|I|}\int_I|tD_{\lambda ,x}P_t^\lambda (f_3)(x)|^2\frac{dxdt}{t}\leq C\|f\|_{BMO_{\rm o}(\mathbb{R})}^2.
$$

Property (\ref{2.1}) is established and we conclude that $\rho _f$ is a Carleson measure on $(0,\infty )\times (0,\infty )$ and
$$
\|\rho _f\|_\mathcal{C}\leq C\|f\|_{BMO_{\rm o}(\mathbb{R})}^2.
$$
Thus the proof of $(i)\Longrightarrow (ii)$ in Theorem \ref{th1.4} is finished.

\section{Proof of $(ii)\Longrightarrow (i)$ in Theorem \ref{th1.4}}\label{sec: (ii)(i)}

We start this section showing the following characterization of $BMO_{\rm o}(\R)$ which we need later. Its proof follows the arguments in \cite[Theorem 1.1]{BCFR1} with minor modifications.

\begin{lema}\label{equiv}
Let $\lambda >0$. Suppose $f\in L^1_{\rm loc}[0,\infty)$. Then, the following assertions are equivalent.
\begin{enumerate}
\item[$(i)$] $f \in BMO_{\rm o}(\R)$.
\item[$(ii)$] $x^\lambda (1+x^2)^{-\lambda -1}f \in L^1(0,\infty)$ and
$$
d\gamma_f(x,t)=\left|t \partial _tP_t^\lambda(f)(x)\right|^2 \frac{dxdt}{t}
$$
is a Carleson measure on $(0,\infty)\times(0,\infty)$.
\end{enumerate}
Moreover, the quantities $\|f\|^2_{BMO_{\rm o}(\R)}$ and $\|\gamma_f\|_{\mathcal{C}}$ are equivalent.
\end{lema}
\begin{proof}
$(i)\Rightarrow (ii)$. It follows from Theorem \ref{th1.3}.

$(ii)\Rightarrow (i)$. We can proceed as in \cite[Section 4]{BCFR1} by establishing the result in \cite[Proposition 4.4]{BCFR1} for the new conditions on $f$. 
Actually, we only have to take into account the following estimations.

Let $a$ be an (odd)-atom, that is, a measurable function satisfying one of the next properties:

(a) $a=\delta ^{-1}\chi _{(0,\delta)}$, for some $\delta >0$;

(b) there exists a bounded interval $I\subset (0,\infty )$ such that supp $a\subset I$, $\int_Ia(x)dx=0$ and $\|a\|_{L^\infty (0,\infty )}\leq |I|^{-1}$.

We have that 
\begin{equation}\label{x1}
\int_0^\infty |t\partial _tP_t^\lambda (y,z)\partial _tP_t^\lambda (a)(y)|dy\leq C\frac{z^\lambda }{(1+z^2)^{\lambda +1}},\quad z,t\in (0,\infty ),
\end{equation}
with $C$ independent of $z$.

Indeed, since
\begin{align*}
\partial_tP_t^\lambda (x,y)&=\frac{2\lambda }{\pi }(xy)^\lambda \left[\int_0^\pi \frac{(\sin \theta)^{2\lambda -1}}{((x-y)^2+t^2+2xy(1-\cos \theta ))^{\lambda +1}}d\theta \right.\\
&\hspace{0.5cm}\left.-2(\lambda +1) t^2\int_0^\pi \frac{(\sin \theta )^{2\lambda -1}}{((x-y)^2+t^2+2xy(1-\cos \theta ))^{\lambda +2}}d\theta \right],\quad x,y,t\in (0,\infty ),
\end{align*}
we get
\begin{equation}\label{3.11*}
\left|\partial _tP_t^\lambda (x,y)\right|\leq C\int_0^\pi \frac{(xy)^\lambda (\sin \theta)^{2\lambda -1}}{((x-y)^2+t^2+2xy(1-\cos \theta ))^{\lambda +1}}d\theta\leq
C\frac{(xy)^\lambda }{((x-y)^2+t^2)^{\lambda +1}},\quad x,y,t\in (0,\infty ),
\end{equation} 
and also
\begin{equation}\label{3.11**}
\left|\partial _tP_t^\lambda (x,y)\right|\leq \frac{C}{(x-y)^2+t^2},\quad x,y,t\in (0,\infty ).
\end{equation} 
Assume that supp $a\subset (0,\alpha)$ for certain $\alpha >0$. Then, 
\begin{align*}
\left|\partial _tP_t^\lambda (a)(y)\right|&\leq \|a\|_{L^\infty (0,\infty )}\int_0^\alpha \frac{(yz)^\lambda }{((y-z)^2+t^2)^{\lambda +1}}dz\\
&\leq Cy^\lambda \left\{\begin{array}{ll}
			\displaystyle t^{-2\lambda -2}\int_0^\alpha z^\lambda dz,&0<y\leq 2\alpha,\\
			& \\
			\displaystyle  (y^2+t^2)^{-\lambda -1}\int_0^\alpha z^\lambda dz,&y\geq 2\alpha,
			\end{array}
\right.\\
&\leq  C\frac{y^\lambda }{(1+y^2)^{\lambda +1}},\quad y,t\in (0,\infty ),
\end{align*}
where $C>0$ does not depend on $y$.
Hence, by using \eqref{3.11*} and \eqref{3.11**} it follows that
\begin{align*}
\int_0^\infty |t\partial _tP_t^\lambda (y,z)\partial _tP_t^\lambda (a)(y)|dy&\\
&\hspace{-4cm}\leq C\left[\left(\int_0^{z/2}+\int_{2z}^\infty \right)\frac{t(yz)^\lambda }{((y-z)^2+t^2)^{\lambda +1}}\frac{
y^\lambda}{(1+y^2)^{\lambda +1}}dy+\int_{z/2}^{2z}\frac{t}{(y-z)^2+t^2}\frac{y^\lambda}{(1+y^2)^{\lambda +1}}dy\right]\\
&\hspace{-4cm}\leq Ctz^\lambda \left[\frac{1}{(z^2+t^2)^{\lambda +1}}\int_0^{z/2}\frac{y^{2\lambda }}{(1+y^2)^{\lambda +1}}dy+\frac{1}{(1+z^2)^{\lambda +1}}\int_{2z}^\infty\frac{y^{2\lambda }}{(y^2+t^2)^{\lambda +1}}dy\right.\\
&\hspace{-3.8cm}\left. +\frac{1}{(1+z^2)^{\lambda +1}}\int_{z/2}^{2z}\frac{dy}{(x-y)^2+t^2}dy\right]\leq C\frac{z^\lambda }{(1+z^2)^{\lambda +1}},\quad z,t\in (0,\infty ).
\end{align*}
Here, the constant $C$ can depend on $t$, but is independent of $z$.

On the other hand we need to estimate $\sup _{t>0}|M_t^\lambda (a)(z)|$, $z\in (0,\infty )$, where
$$
M_t^\lambda (a)=\frac{1}{4}\Big[t\partial _vP_{2v}^\lambda (a)_{|v=t}-P_{2t}^\lambda (a)\Big],\quad t\in (0,\infty ).
$$

According to \cite[p. 492]{BCFR1} we have that
$$
\sup _{t>0}|M_t^\lambda (a)(z)|\leq C\left\{\begin{array}{ll}
						1,&0<z\leq 2\alpha ,\\
						z^{-\lambda -2},&z\geq  2\alpha.
						\end{array}
\right.
\leq C\left\{\begin{array}{ll}
						1,&0<z\leq 2\alpha ,\\
						\displaystyle \frac{z^\lambda }{(1+z^2)^{\lambda +1}},&z\geq  2\alpha ,
						\end{array}
\right.
$$
which allows us to obtain
\begin{equation}\label{x2}
\int_0^\infty |f(z)|\sup _{t>0}|M_t^\lambda (a)(z)|dz\leq \int_0^{2\alpha }|f(z)|dz+\int_{2\alpha }^\infty \frac{z^\lambda |f(z)|}{(1+z^2)^{\lambda +1}}dz<\infty .
\end{equation}

By using \eqref{x1} and \eqref{x2} and proceeding as in \cite[Section 4]{BCFR1} we conclude our result.
\end{proof}

Assume that $u$ is a $\lambda$-harmonic function on $(0,\infty )\times (0,\infty )$ such that $x^{-\lambda} u(x,t)\in C^\infty (\mathbb{R}\times (0,\infty ))$ is even in the $x$-variable and that the measure
$$
d\mu _\lambda (x,t)=|t\nabla _\lambda u(x,t)|^2\frac{dxdt}{t}
$$
is Carleson on $(0,\infty )\times (0,\infty )$.

The function $u$ satisfies the equation
$$
\partial_t^2u+\partial _x^2u-\frac{\lambda (\lambda -1)}{x^2}u=0,
$$
in a weak sense on $\mathbb{R}\times (0,\infty )$, that is, for every $\phi \in C_c^\infty (\mathbb{R}\times (0,\infty ))$, the space of smooth functions having compact support on $\mathbb{R}\times (0,\infty )$,
\begin{equation}\label{3.1}
0=\int_{\mathbb{R}\times (0,\infty )}\left(\partial _tu(x,t)\partial_t\phi (x,t)+\partial_ xu(x,t)\partial_x\phi (x,t)+\frac{\lambda (\lambda -1)}{x^2}u(x,t)\phi (x,t)\right)dxdt.
\end{equation}

Indeed, let $\phi \in C_c^\infty (\mathbb{R}\times (0,\infty ))$. We choose $0<a<\infty$ and $0<b_1<b_2<\infty$ such that ${\rm supp}(\phi )\subset [-a,a]\times [b_1,b_2]$ and define $v(x,t)=x^{-\lambda} u(x,t)$, $(x,t)\in \mathbb{R}\times (0,\infty )$. 

Since $v\in C^2(\mathbb{R}\times (0,\infty ))$ and $\lambda >1$, we have that $\frac{u}{x^2}$, $\partial_x u$ and  $\partial _ tu$ are in $L^1_{\rm loc}(\mathbb{R}\times (0,\infty ))$, and $\lim_{x\rightarrow 0} \partial_xu(x,t)=0$, for every $t\in (0,\infty )$.
Moreover, 
$$
\partial _t^2u(x,t)+\partial _x^2u(x,t)-\frac{\lambda (\lambda -1)}{x^2}u(x,t)=0, \quad (x,t)\in (\mathbb{R}\setminus \{0\})\times (0,\infty ).
$$

Then, we can write
 \begin{align*}
      \int_{\mathbb{R}\times (0,\infty )}\left(\partial_tu(x,t)\partial_t\phi (x,t)+\partial_xu(x,t) \partial_x\phi (x,t)+\frac{\lambda (\lambda -1)}{x^2}u(x,t)\phi (x,t)\right)dxdt\\
            &\hspace{-11cm}=\lim_{\varepsilon \rightarrow 0^+}\int_{b_1}^{b_2}\left(\int_{-a}^{-\varepsilon}+\int_\varepsilon ^a\right)\left(\partial_tu(x,t) \partial_t\phi (x,t)+ \partial_ xu(x,t)\partial_x\phi (x,t)+\frac{\lambda (\lambda -1)}{x^2}u(x,t)\phi (x,t)\right)dxdt\\
 &\hspace{-11cm}=\lim_{\varepsilon \rightarrow 0^+}\int_{b_1}^{b_2}\left(\int_{-a}^{-\varepsilon}+\int_\varepsilon ^a\right)\left(-\partial_t^2u(x,t)- \partial_x^2u(x,t)+\frac{\lambda (\lambda -1)}{x^2}u(x,t)\right)\phi (x,t)dxdt=0.
    \end{align*}
Since (\ref{3.1}) holds, by proceeding as in \cite[Lemma 2.6]{DYZ} (see also \cite[Lemma 2.1]{Sh}) we can prove that the function $u^2$ is subharmonic in $\mathbb{R}\times (0,\infty )$. Hence, for every $x_0\in \mathbb{R}$, $t_0\in (0,\infty )$ and $0<r<t_0$,
$$
|u(x_0,t_0)|\leq \left(\frac{1}{\pi r^2}\int_{B((x_0,t_0),r)}|u(x,t)|^2dxdt\right)^{1/2}.
$$
It is clear that $\partial _tu$ satisfies the same properties than $u$. Then, for every $x_0\in \mathbb{R}$, $t_0\in (0,\infty )$ and $0<r<t_0$,
$$
|\partial_tu(x_0,t_0)|\leq \left(\frac{1}{\pi r^2}\int_{B((x_0,t_0),r)}|\partial_tu(x,t)|^2dxdt\right)^{1/2}.
$$

Since the measure $t|\partial _tu(x,t)|^2dxdt$ is Carleson on $(0,\infty )\times (0,\infty )$ we have that, for every $x_0,t_0\in (0,\infty )$,
\begin{align}\label{3.2}
|(\partial_s u(x_0,s))_{|s=t_0}|&\leq C\left(\frac{1}{t_0^2}\int_{B((x_0,t_0),t_0/2)}|\partial_tu(x,t)|^2dxdt\right)^{1/2}\leq C\left(\frac{1}{t_0^2}\int_{t_0/2}^{3t_0/2}\int_{x_0-t_0/2}^{x_0+t_0/2}|\partial_tu(x,t)|^2dxdt\right)^{1/2}\nonumber\\
&\leq \frac{C}{t_0}\left(\frac{1}{t_0}\int_0^{3t_0/2}\int_{I(x_0,t_0)}t|\partial_tu(x,t)|^2dxdt\right)^{1/2}\leq \frac{C}{t_0}\|t|\partial_tu(x,t)|^2dxdt\|_{\mathcal{C}}^{1/2},
\end{align}
where $I(x_0,t_0)=(x_0-\frac{3t_0}{4},x_0+\frac{3t_0}{4})\cap (0,\infty )$. We have used that $|\partial_tu(x,t)|=|\partial_tu(-x,t)|$, $x\in \mathbb{R}$ and $t\in (0,\infty )$.

From (\ref{3.2}) we deduce that, for every $t_0>0$, there exists $C>0$ such that 
\begin{equation}\label{9.1}
|\partial_tu(x,t)|\leq C, \quad x\in \mathbb{R} \mbox{ and }t\geq t_0.
\end{equation}

Our next objective is to show that, for every $t_0>0$,
\begin{equation}\label{3.3}
\partial_t u(x,t+t_0)=P_t^\lambda ((\partial _s u (\cdot , s))_{|s=t_0})(x),\quad x,t\in (0,\infty ).
\end{equation}

In order to see this property we establish previously some results.
\begin{lema}\label{Lem3.1}
Let $\lambda >0$. Suppose that $f$ is a continuous function on $(0,\infty )$ such that 
$$
\int_0^\infty \frac{y^\lambda |f(y)|}{(1+y^2)^{\lambda +1}}dy<\infty .
$$
Then, the function
$$
v(x,t)=\left\{\begin{array}{ll}
		P_t^\lambda (f)(x),&x,t\in (0,\infty ),\\
		f(x),&x\in (0,\infty ),\;t=0
		\end{array}
\right.
$$
is $\lambda$-harmonic in $(0,\infty )\times (0,\infty )$ and continuous in $(0,\infty )\times [0,\infty )$.
\end{lema}
\begin{proof}
Differentiating under the integral sign and using \cite[(16.1')]{MS} it is not hard to see that $v$ is $\lambda$-harmonic function on $(0,\infty )\times (0,\infty )$.

Suposse firstly that $f$ is bounded in $(0,\infty )$. Let $x_0\in (0,\infty )$. We write the following decomposition
\begin{align*}
P_t^\lambda (f)(x)-f(x_0)&=\int_0^\infty P_t^\lambda (x,y)[f(y)-f(x_0)]dy+\left(\int_0^\infty P_t^\lambda (x,y)dy-1\right)f(x_0)\\
&=:I_1(x,t)+I_2(x,t),\quad x,t\in (0,\infty ).
\end{align*}
Assume that $\varepsilon >0$. There exists $\delta \in (0,x_0/2)$ such that $|f(y)-f(x_0)|<\varepsilon $ provided that $|y-x_0|<\delta $, because $f$ is continuous in $x_0$. Since $f$ is bounded in $(0,\infty )$ we get
\begin{align*}
|I_1(x,t)|&\leq \left(\int_{|y-x_0|<\delta }+\int_{|y-x_0|\geq\delta }\right)P_t^\lambda (x,y)|f(y)-f(x_0)|dy\\
&\leq \varepsilon \int_{|y-x_0|<\delta }P_t^\lambda (x,y)dy+2\|f\|_{L^\infty (0,\infty )}\int_{|y-x_0|\geq\delta }P_t^\lambda (x,y),\quad x,t\in (0,\infty ).
\end{align*}
By \cite[p. 86, (b)]{MS} we obtain
$$
 \int_{|y-x_0|<\delta }P_t^\lambda (x,y)dy\leq C\int_{-\infty}^{+\infty}\frac{t}{(x-y)^2+t^2}dy\leq C,\quad x,t\in (0,\infty ),
$$
and
\begin{align*}
 \int_{|y-x_0|\geq\delta }P_t^\lambda (x,y)dy&\leq C \int_{|y-x_0|\geq\delta }\frac{t}{(x-y)^2+t^2}dy\leq Ct \int_{|y-x_0|\geq\delta }\frac{dy}{(x-y)^2}\\
&\leq  Ct \int_{|y-x|\geq\delta /2}\frac{dy}{(x-y)^2}\leq C\frac{t}{\delta},\quad |x-x_0|<\frac{\delta}{2}\mbox{ and }t>0.
\end{align*}
Hence,
\begin{equation}\label{3.4}
|I_1(x,t)|\leq C\left(\varepsilon +\frac{t}{\delta}\right),\quad |x-x_0|<\frac{\delta}{2}\mbox{ and }t>0.
\end{equation}

On the other hand, by taking into account that $\int_0^\infty x^{-\lambda }y^\lambda P_t^\lambda (x,y)dy=1$, $x,t\in (0,\infty )$, (see, \cite[p. 29 (4)]{EMOT}, \cite[\S 2 (1), (2)]{Hirs} and \cite[(16.1')]{MS}), we get
$$
\left|\int_0^\infty P_t^\lambda (x,y)dy-1\right|\leq \int_0^\infty \left|1-\Big(\frac{y}{x}\Big)^\lambda \right|P_t^\lambda (x,y)dy.
$$
We choose $\eta\in (0,1)$ such that $|1-z^\lambda |<\varepsilon $ provided that $|1-z|<\eta$. From \cite[p. 86, (b)]{MS} we deduce that
\begin{align*}
\left|\int_0^\infty P_t^\lambda (x,y)dy-1\right|&\leq \left(\int_0^{(1-\eta)x}+\int_{(1-\eta )x}^{(1+\eta )x}+\int_{(1+\eta )x}^\infty \right)\left|1-\Big(\frac{y}{x}\Big)^\lambda \right|P_t^\lambda (x,y)dy\\
&\leq C\left(\int_0^{(1-\eta )x}\frac{(1+(1-\eta)^\lambda )t}{(x-y)^2+t^2}dy+\varepsilon \int_{(1-\eta )x}^{(1+\eta)x}\frac{t}{(x-y)^2+t^2}dy\right.\\
&\hspace{0.1cm}\left. +\int_{(1+\eta)x}^\infty  \left(\Big(\frac{y}{x}\Big)^\lambda- 1\right)\frac{t(xy)^\lambda}{((x-y)^2+t^2)^{\lambda +1}}dy\right)\\
&\leq C\left(\frac{(1+(1-\eta )^\lambda)t}{(\eta x)^2}\int_0^{(1-\eta)x}dy+\varepsilon +t\int_{(1+\eta )x}^\infty \frac{(y^\lambda -x^\lambda)y^\lambda }{(\frac{\eta y}{1+\eta})^{2\lambda +2}}dy\right)\\
&\leq C\left(\frac{t}{\eta ^2x}+\varepsilon +\frac{t}{\eta ^{2\lambda +2}}\int_{(1+\eta)x}^\infty \frac{dy}{y^2}\right)\leq  C\left(\varepsilon +\frac{t}{\eta ^{2\lambda +2}x}\right)\\
&\leq C\left(\varepsilon +\frac{t}{\eta ^{2\lambda +2}}\right) ,\quad |x-x_0|<\frac{x_0}{2}\mbox{ and }t>0.
\end{align*}
Then
\begin{equation}\label{3.5}
|I_2(x,t)|\leq C\left(\varepsilon +\frac{t}{\eta ^{2\lambda +2}}\right),\quad |x-x_0|<\frac{x_0}{2}\mbox{ and }t>0.
\end{equation}
Putting together (\ref{3.4}) and (\ref{3.5}) we conclude that
$$
\lim_{\substack{(x,t)\rightarrow (x_0,0)\\
 x,t\in (0,\infty )}}P_t^\lambda (f)(x)=f(x_0).
$$

We now study the general case, that is, consider $f$ a continuous function such that
$$
\int_0^\infty \frac{y^\lambda |f(y)|}{(1+y^2)^{\lambda +1}}dy<\infty .
$$
Let $x_0\in (0,\infty )$. For every $n\in \mathbb{N}$ we denote by $\phi _n$ a smooth function on $(0,\infty )$ such that $\phi _n(x)=1$, $x\in (1/n,n)$, and $\phi _n(x)=0$, $x\in (0,\infty )\setminus (1/(n+1),n+1)$. 

Suppose that $\varepsilon >0$ and let $n_0\in \mathbb{N}$ such that $x_0\in (1/n_0, n_0)$. We can write
\begin{align}\label{C1}
|P_t^\lambda (f)(x)-f(x_0)|&\leq |P_t^\lambda (f-f\phi _n)(x)|+|P_t^\lambda (f\phi _n)(x)-(f\phi _n)(x_0)|+|(f\phi _n)(x_0)-f(x_0)|\nonumber \\
&=|P_t^\lambda (f-f\phi _n)(x)|+|P_t^\lambda (f\phi _n)(x)-(f\phi _n)(x_0)|,\quad n\geq n_0.
\end{align}

According to \cite[p. 86 (b)]{MS} we have that, for each $|x-x_0|<x_0/2$, $t\in (0,1)$ and $n\in \mathbb{N}, n\geq 4n_0$,
\begin{align*}
|P_t^\lambda (f-f\phi _n)(x)|&\leq Ctx^\lambda \left(\int_0^{1/n}+\int_n^\infty \right)\frac{y^\lambda |f(y)|}{((x-y)^2+t^2)^{\lambda +1}}dy\\
&\leq  Cx_0^\lambda \left(\int_0^{1/n}\frac{y^\lambda |f(y)|}{(x_0^2+t^2)^{\lambda +1}}dy+ \int_n^\infty \frac{y^\lambda |f(y)|}{(y^2+t^2)^{\lambda +1}}dy\right)\\
&\leq C\left(\frac{1}{x_0^{\lambda +2}}\int_0^{1/n}y^\lambda |f(y)|dy+x_0^\lambda \int_n^\infty \frac{y^\lambda |f(y)|}{y^{2\lambda +2}}dy\right)\\
&\leq C\left(\frac{1}{x_0^{\lambda +2}}\int_0^{1/n}\frac{y^\lambda |f(y)|}{(1+y^2)^{\lambda +1}}dy+x_0^\lambda \int_n^\infty \frac{y^\lambda |f(y)|}{(1+y^2)^{\lambda +1}}dy\right),
\end{align*}
with $C$ independent of $x$, $t$ and $n$.

Then, we can find $n_1\in \mathbb{N}$, $n_1\geq 4n_0$, such that
\begin{equation}\label{C2}
|P_t^\lambda (f-f\phi _n)(x)| <\varepsilon , \quad |x-x_0|<\frac{x_0}{2},\;t\in (0,1), \;n\in \mathbb{N},\;n\geq n_1.
\end{equation}

On the other hand, for each $n\in \mathbb{N}$, since $f\phi _n$ is continuous and bounded on $(0,\infty )$, 
\begin{equation}\label{C3}
\lim_{\substack{(x,t)\rightarrow (x_0,0)\\
		      x,t\in (0,\infty )}}P_t^\lambda (f\phi _n)(x)=(f\phi _n)(x_0).
\end{equation}

By considering (\ref{C1}), (\ref{C2}) and (\ref{C3}) we conclude that
$$
\lim_{\substack{(x,t)\rightarrow (x_0,0)\\
			x,t\in (0,\infty )}}P_t^\lambda (f)(x)=f(x_0).
$$
\end{proof}
The space of $\lambda$-harmonic functions on $ (0,\infty )\times\mathbb{R}$ form a Brelot harmonic space. Then, it is well-known that $\lambda$-harmonic functions on $(0,\infty )\times \mathbb{R}$ satisfy the mean value properties with respect to the $\lambda$-harmonic measures. Recently, Eriksson and Orelma (\cite{EO}) have established explicit mean value properties for solutions of Weinstein operators. We recall some results in \cite{EO} specified for our particular case and that will be useful.

We consider on $(0,\infty )\times \mathbb{R}$ the hyperbolic metric $d_h$ defined by
$$
d_h(a,b)={\rm arcosh \;}\sigma (a,b),\quad a,b\in  (0,\infty )\times\mathbb{R},
$$
where
$$
\sigma (a,b)=\frac{(a_1-b_1)^2+(a_2-b_2)^2+2a_1b_1}{2a_1b_1},\quad a=(a_1,b_1), b=(b_1,b_2)\in (0,\infty )\times \mathbb{R}.
$$
The hiperbolic ball $B_h(a,r)$ with center $a\in (0,\infty )\times \mathbb{R}$ and radius $r>0$ is defined as usual by
$$
B_h(a,r)=\{b\in (0,\infty )\times \mathbb{R}: d_h(a,b)<r\}.
$$
For every $a\in (0,\infty )\times \mathbb{R}$ and $r>0$, $B_h(a,r)$ is actually an Euclidean ball. We have that, for each $a=(a_1,a_2)\in (0,\infty )\times \mathbb{R}$ and $r>0$
$$
B_h(a,r)=\{b\in (0,\infty )\times \mathbb{R}: |\tilde{a}-b|<a_1\sinh r\},
$$
where $\tilde{a}=(a_1\cosh r,a_2)$.

In \cite{AL} Akin and Leutwiler introduced the function
$$
\varphi _\alpha (r)=\frac{(1-r^2)^\alpha }{2}\int_{-1}^1\frac{dy}{|r-y|^{2\alpha}},\quad 0<r<1,
$$
in their investigations about Weinstein equations.

From \cite[Theorem 3.3]{EO} it follows the following mean value property for $\lambda$-harmonic functions.
\begin{lema}\label{Lem3.2}
Let $\lambda >0$. Assume that $U$ is an open subset of $(0,\infty )\times \mathbb{R}$. If $v$ is a $\lambda$-harmonic function in $U$ then, for every $a\in U$ and $r>0$ such that $\overline{B_h(a,r)}\subset U$,
\begin{equation}\label{3.6}
v(a)=\frac{1}{2\sinh (r)\varphi _\alpha (\tanh(r/2))}\int_{\partial B_h(a,r)}v(b_1,b_2)\frac{d\tau (b_1,b_2)}{b_1},
\end{equation}
where $\alpha =(1+|2\lambda -1|)/2$ and $\tau$ denotes the length measure on $\partial B_h(a,r)$.
\end{lema}

We now prove the converse of Lemma \ref{Lem3.2}.
\begin{lema}\label{Lem3.3}
Let $\lambda >0$ and let $U$ be an open subset of $(0,\infty )\times \mathbb{R}$. Suppose that $v$ is a continuous function on $U$ such that the mean value property (\ref{3.6}) holds for every $a\in U$ and $r>0$ such that $\overline{B_h(a,r)}\subset U$. Then, $v$ is $\lambda $-harmonic in $U$.
\end{lema}

\begin{proof}
In order to show this property we follow a procedure similar to the classical one used to establish the corresponding result for harmonic functions.

In a first step we prove a maximum principle in this context. Let $a\in U$ and $r>0$ such that $\overline{B_h(a,r)}\subset U$. Since $v$ is continuous in $\overline{B_h(a,r)}$, the set
$$
A=\{b\in \overline{B_h(a,r)}: v(b)\geq v(c), \;c\in B_h(a,r)\}\not=\emptyset.
$$
Suppose that $A\cap \partial B_h(a,r)=\emptyset$. Since $A$ is closed, 
$$
d(A,\partial B_h(a,r))=\min \{|c-z|: c\in A, z\in \partial B_h(a,r)\}>0.
$$
We choose $b\in A$ such that 
$$
d(b,\partial B_h(a,r))=\inf \{|b-z|: z\in \partial B_h(a,r)\}=d(A,\partial B_h(a,r))
$$
and $R>0$ such that $B_h(b,R)\subset B_h(a,r)$. We consider the sets
$$
M_+=A\cap \partial B_h(b,R)\quad \mbox{ and }M_-=A^{\rm c}\cap \partial B_h(b,R).
$$
Since $\tau (M_-)>0$ we deduce that
\begin{align*}
\frac{1}{2\sinh (R)\varphi _\alpha (\tanh (R/2))}\int_{\partial B_h(b,R)}v(z_1,z_2)\frac{d\tau (z_1,z_2)}{z_1}&\\
&\hspace{-4cm}=
\frac{1}{2\sinh (R)\varphi _\alpha (\tanh (R/2))}\left(\int_{M_+}+\int_{M_-}\right)v(z_1,z_2)\frac{d\tau (z_1,z_2)}{z_1}<v(b).
\end{align*}
We have taken into account that
$$
\int_{\partial B_h(b,R)}\frac{d\tau (z_1,z_2)}{z_1}=2\sinh(R)\varphi _\alpha \Big(\tanh\frac{R}{2}\Big).
$$
Hence, since $v$ satisfies (\ref{3.6}) for every $a\in U$ and $r>0$ such that $\overline{B_h(a,r)}\subset U$, $A\cap \partial B_h(a,r)\not=\emptyset$. Then,
$$
\max _{b\in \overline{B_h(a,r)}} v(b)=\max_{b\in \partial B_h(a,r)}v(b).
$$
We now observe that the  operator
$$
\mathcal{L}_\lambda =\partial_t^2+\partial _x^2-\frac{\lambda (\lambda -1)}{x^2},
$$
is uniformly elliptic on every bounded domain $\Omega$ such that $\overline{\Omega}\subset (0,\infty )\times \mathbb{R}$. Then, for every $b\in U$ and $R>0$ such that $\overline{B_h(b,R)}\subset U$ and every continuous function $f$ on $\partial B_h(b,R)$, there exists a continuous function $w$ in $\overline{B_h(b,R)}$ such that $w_{|\partial B_h(b,R)}=f$ and $w$ is $\lambda$-harmonic in $B_h(b,R)$. Hence, according to Lemma \ref{Lem3.2}, this function $w$ satisfies the mean value property (\ref{3.6}) for every $a\in B_h(b,r)$ and $ r>0$ such that $\overline{B_h(a,r)}\subset B_h(b,R)$.

Let $b\in U$ and $R>0$ such that $\overline{B_h(b,R)}\subset U$. We define $f=v_{|\partial B_h(b,R)}$ and denote by $w$ the continuous function in $\overline{B_h(b,R)}$ such that $w_{|\partial B_h(b,R)}=f$ and $w$ is $\lambda$-harmonic in $B_h(b,R)$. We consider the function $F=v-w$ in $\overline{B_h(b,R)}$. It is clear that $F_{|\partial B_h(b,R)}=0$ and $F$ satisfies the mean value property (\ref{3.6}) for every $a\in B_h(b,R)$ and $r>0$ such that $\overline{B_h(a,r)}\subset B_h(b,R)$. The maximum (minimum) property allows us to conclude that $v=w$ in $\overline{B_h(b,R)}$. Thus, we prove that $v$ is $\lambda$-harmonic in $U$.
\end{proof}

{\sc Remark} As it can be deduced from the proof of Lemma \ref{Lem3.3} , in order to see that a function $v$ continuous in an open subset $U$ of $(0,\infty )\times \mathbb{R}$ is $\lambda$-harmonic in $U$, it is sufficient to show that, for every $a\in U$, there exists a sequence $(r_n)_{n\in \mathbb{N}}\subset (0,\infty )$ such that, $r_n\longrightarrow 0$, as $n\rightarrow \infty $,  that $\overline{B_h(a,r_n)}\subset U$, $n\in \mathbb{N}$, and
$$
v(a)=\frac{1}{2\sinh (r_n)\varphi _\alpha (\tanh(r_n/2))}\int_{\partial B_h(a,r_n)}v(b_1,b_2)\frac{d\tau (b_1,b_2)}{b_1},
$$
with $\alpha =(1+|2\lambda -1|)/2$.

Now we establish a uniqueness result for $\lambda$-harmonic functions in $(0,\infty)\times (0,\infty )$.
\begin{lema}\label{Lem3.4}
Let $\lambda >1$. Suppose that $v$ is a bounded and continuous function on $(0,\infty )\times [0,\infty )$ such that $v$ is $\lambda$-harmonic in $(0,\infty )\times (0,\infty)$ and $v(x,0)=0$, $x\in (0,\infty )$. Then, $v=0$ in $(0,\infty )\times [0,\infty )$.
\end{lema}
\begin{proof}
We define
$$
w(x,t)=\left\{\begin{array}{ll}
		v(x,t),&x\in (0,\infty), \;t\in [0,\infty)\\
		-v(x,-t),&x\in (0,\infty ),\;t\in (-\infty, 0).\\
		\end{array}
\right.
$$
$w$ is a continuous function in $(0,\infty )\times \mathbb{R}$. Moreover, $w$ is $\lambda$-harmonic in $(0,\infty )\times \mathbb{R}\setminus\{0\}$. According to Lemma \ref{Lem3.3}, in order to see that $w$ is $\lambda$-harmonic in $(0,\infty )\times \mathbb{R}$ it is sufficient to observe that, for every $x\in (0,\infty )$ and $r>0$ such that $\overline{B_h((x,0),r)}\subset (0,\infty )\times \mathbb{R}$,
$$
0=\int_{\partial B_h((x,0), r)}w(b_1,b_2)\frac{d\tau (b_1,b_2)}{b_1}.
$$
Note that this property holds because $w$ is odd in the second variable and every hyperbolic ball centered in the line $(0,\infty)\times \{0\}$ is actually an Euclidean ball with center in the same line.

Since $v$ is bounded in $(0,\infty )\times [0,\infty )$, $w$ is also bounded in $(0,\infty )\times \mathbb{R}$. Then, there  exists $M>0$ such that $|w(x,t)|\leq M$, $x\in (0,\infty )$ and $t\in \mathbb{R}$. The function $g(x,t)=x^\lambda +x^{1-\lambda}$, $x\in (0,\infty )$ and $t\in \mathbb{R}$, is $\lambda$-harmonic in $(0,\infty )\times \mathbb{R}$. We define the function
$$
\tilde{w}(x,t)=w(x,t)+M(x^\lambda +x^{1-\lambda}),\quad x\in (0,\infty )\mbox{ and }t\in \mathbb{R}.
$$
Thus, $\tilde{w}(x,t)\geq 0$, $x\in (0,\infty )$ and $t\in \mathbb{R}$, and $\tilde{w}$ is $\lambda$-harmonic in $(0,\infty )\times \mathbb{R}$. According to \cite[Theorem 2.2]{Leu} there exists a positive $\sigma$-finite measure $\gamma$ on $\mathbb{R}$ and $m\geq 0$ such that
$$
\tilde{w}(x,t)=x^\lambda \left(m+\int_{-\infty}^{+\infty}\frac{d\gamma (s)}{((t-s)^2+x^2)^\lambda}\right),\quad x\in (0,\infty )\mbox{ and }t\in \mathbb{R}.
$$
Then,
$$
w(x,t)=-M(x^\lambda +x^{1-\lambda})+x^\lambda \left(m+\int_{-\infty}^{+\infty}\frac{d\gamma (s)}{((t-s)^2+x^2)^\lambda}\right),\quad x\in (0,\infty )\mbox{ and }t\in \mathbb{R}.
$$
Since $w(x,0)=0$, $x\in (0,\infty )$, we have that
$$
-Mx^{1-2\lambda}+m-M+\int_{-\infty }^{+\infty }\frac{d\gamma (s)}{(s^2+x^2)^\lambda }=0,\quad x\in (0,\infty ).
$$
By letting $x\rightarrow +\infty$ and by dominated convergence theorem we deduce that $m=M$. Hence,
\begin{equation}\label{3.7*}
w(x,t)=x^{1-\lambda }\left(-M+\int_{-\infty }^{+\infty }\frac{x^{2\lambda -1}}{((t-s)^2+x^2)^\lambda }d\gamma (s)\right),\quad x\in (0,\infty )\mbox{ and }t\in \mathbb{R},
\end{equation}
and again, since $w(x,0)=0$, $x\in (0,\infty )$, we deduce that
\begin{equation}\label{3.7}
M=\int_{-\infty }^{+\infty }\frac{x^{2\lambda-1}}{(s^2+x^2)^\lambda }d\gamma (s),\quad x\in (0,\infty ).
\end{equation}
By using Radon-Nikodym theorem we can write $d\gamma (s)=hds+d\mu (s)$, where $0\leq h\in L^1_{\rm loc}(\mathbb{R})$ and $\mu$ is a positive measure that is orthogonal to the Lebesgue measure on $\mathbb{R}$.

It can be seen that
\begin{equation}\label{limite}
\lim_{x\rightarrow 0^+}\int_{-\infty }^{+\infty }\frac{x^{2\lambda -1}}{((t-s)^2+x^2)^\lambda }d\gamma (s)=Ah(t),\quad \mbox{a.e. }t\in \mathbb{R}.
\end{equation}
Here, a.e. is understood with respect to the Lebesgue measure on $\mathbb{R}$ and 
$$
A=\int_{-\infty}^{+\infty }\frac{1}{(s^2+1)^{\lambda }}ds=\frac{\sqrt{\pi}\Gamma (\lambda -1/2)}{\Gamma (\lambda)}.
$$
Indeed, fix $N\in \mathbb{N}$. It is sufficient to see (\ref{limite}) for a.e. $|t|\leq N$. Denote by $K_x$, $x\in (0,\infty )$, the kernel
$$
K_x(t,s)=\frac{x^{2\lambda -1}}{((t-s)^2+x^2)^\lambda },\quad t,s\in \mathbb{R}.
$$
For every $n\in \mathbb{N}$, let us define $h_n(t)=h(t)\chi _{(-n,n)}(t)$, $t\in \mathbb{R}$. Then, since $\int_{-\infty  }^{+\infty }K_x(t,s)ds=A$, $x\in (0,\infty ),t\in \mathbb{R}$, it follows that, for each $n\in \mathbb{N}$, $n\geq N$, we can write 
\begin{align}\label{des}
\int_{-\infty}^{+\infty}K_x(t,s)d\gamma (s)-Ah(t)&=\int_{-\infty}^{+\infty}K_x(t,s)[h(s)-h_n(s)]ds+\int_{-\infty}^{+\infty}K_x(t,s)h_n(s)ds-Ah_n(t)\nonumber\\
&\hspace{0.3cm}+\int_{-\infty}^{+\infty}K_x(t,s)d\mu (s),\quad x\in (0,\infty ), \;|t|\leq N.
\end{align}

When $n\geq 2N$, the first term can be bounded as follows,
\begin{align*}
\left|\int_{-\infty}^{+\infty}K_x(t,s)[h(s)-h_n(s)]ds\right|&\leq \int_{|s|>n}\frac{x^{2\lambda -1}|h(s)|}{((t-s)^2+x^2)^\lambda }ds\leq C\int_{|s|>n}\frac{|h(s)|}{(s^2+x^2)^\lambda }ds\\
&\leq C\int_{|s|>n}\frac{|h(s)|}{s^{2\lambda }}ds\leq C\int_{|s|>n}\frac{|h(s)|}{(s^2+1)^\lambda }ds,\quad x\in (0,1),\;|t|\leq N.
\end{align*}

Thus, for every $\varepsilon >0$, there exists $n_0\in \mathbb{N}$, $n_0\geq 2N$, independent of $x\in (0,1)$ and $|t|\leq N$, such that
\begin{equation}\label{a}
\left|\int_{-\infty}^{+\infty}K_x(t,s)[h(s)-h_{n_0}(s)]ds\right|<\varepsilon, \quad x\in (0,1),\;|t|\leq N.
\end{equation}

On the other hand, we observe that
$$
|K_x(t,s)|\leq C\left\{\begin{array}{ll}
			\displaystyle \frac{1}{x},&|t-s|<x,\\
			&\\
			\displaystyle \frac{1}{2^{(2\lambda -1)k}2^kx},&2^{k-1}x\leq |t-s|<2^kx,
			\end{array}
\right.,\quad x\in (0,\infty ), \;t,s\in \mathbb{R},\;k\in \mathbb{N}.
$$
Then, since $\lambda >1$, it is not difficult to see that
$$
\sup_{x\in (0,\infty )}\left|\int_{-\infty}^{+\infty}K_x(t,s)h_{n_0}(s)ds\right|\leq C\mathcal{M}(|h_{n_0}|)(t),\quad t\in \mathbb{R},
$$
and
$$
\sup_{x\in (0,\infty )}\left|\int_{-\infty}^{+\infty}K_x(t,s)d\mu(s)ds\right|\leq C\mathcal{M}(\mu)(t),\quad t\in \mathbb{R},
$$
where $\mathcal{M}$ represents the classical Hardy-Littlewood maximal function defined on $L^1(\mathbb{R})$ and on the set of the Borel measures on $\mathbb{R}$.

By following standard arguments (see \cite[Theorems 6.39 and 6.42]{ABR}, for ins\-tan\-ce)  we obtain that
\begin{equation}\label{b}
\lim_{x\rightarrow 0^+}\int_{-\infty }^{+\infty }K_x(t,s)h_{n_0}(s)ds=Ah_{n_0}(t),\quad \mbox{a.e. }t\in \mathbb{R},
\end{equation}
and
\begin{equation}\label{c}
\lim_{x\rightarrow 0^+}\int_{-\infty }^{+\infty }K_x(t,s)d\mu(s)=0, \quad \mbox{a.e. }t\in \mathbb{R}.
\end{equation}

Putting together (\ref{des}), (\ref{a}), (\ref{b}) and (\ref{c}) we obtain (\ref{limite}) for a.e. $|t|\leq N$. 

By taking into account that $w$ is a bounded function in $(0,\infty )\times \mathbb{R}$ and $\lambda >1$,  from (\ref{3.7*}) we deduce that
$$
-M+Ah(t)=0,\quad \mbox{a.e. }t\in \mathbb{R},
$$
and by (\ref{3.7}), it follows that
$$
\int_{-\infty }^{+\infty}\frac{d\mu (s)}{(s^2+x^2)^\lambda}=0,\quad x\in (0,\infty ).
$$
Hence, $\mu =0$. By using again (\ref{3.7*}) we obtain
$$
w(x,t)=x^{1-\lambda}\left(-M+\frac{M}{A}\int_{-\infty }^{+\infty}\frac{x^{2\lambda -1}}{((t-s)^2+x^2)^\lambda }ds\right)=0,\quad x\in (0,\infty )\mbox{ and }t\in \mathbb{R}.
$$
Then $v(x,t)=0$, $x\in (0,\infty )$ and $t\geq 0$.
\end{proof}

\begin{proof}[Proof of (\ref{3.3})] Let $t_0>0$. We define the function $v(x,t)=\partial_tu(x,t+t_0)$, $x\in (0,\infty )$ and $t\in [0,\infty )$. We have that $v$ is bounded (see (\ref{9.1})),  continuous in $(0,\infty )\times [0,\infty )$ and $\lambda$-harmonic in $(0,\infty )\times (0,\infty )$. We consider $f(x)=v(x,0)$, $x\in (0,\infty )$, and define
$$
V(x,t)=\left\{\begin{array}{ll}
		P_t^\lambda (f)(x),& x,t\in (0,\infty ),\\
		f(x),&x\in (0,\infty )\mbox{ and }t=0.
		\end{array}
\right.
$$
Since $f$ is bounded and continuous in $(0,\infty )$, by Lemma \ref{Lem3.1}, the function $V$ is continuous and bounded in $(0,\infty )\times [0,\infty)$ and $\lambda $-harmonic in $(0,\infty )\times (0,\infty )$. The function $V-v$ is bounded and continuous in $(0,\infty )\times [0,\infty )$, and $\lambda$-harmonic in $(0,\infty )\times (0,\infty )$. Moreover, $V(x,0)=v(x,0)$, $x\in (0,\infty )$. According to Lemma \ref{Lem3.4}, $V(x,t)=v(x,t)$, $x\in (0,\infty )$ and $t\in [0,\infty )$. Thus, (\ref{3.3}) is established.
\end{proof}

Our next objective is to establish that
\begin{equation}\label{3.8}
u(x,t+r)=P_t^\lambda (u(\cdot , r))(x),\quad x,t,r\in (0,\infty ).
\end{equation}

We have that, for every $r>0$,
\begin{equation}\label{M1}
\int_0^\infty \frac{y^\lambda |u(y,r)|}{(1+y^2)^{\lambda +1}}dy<\infty ,
\end{equation}
and then the integral defining $P_t^\lambda (u(\cdot , r))(x)$ is absolutely convergent, for every $x,t\in (0,\infty )$.

In order to show (\ref{3.8}) we see previously that
\begin{equation}\label{3.9}
\lim_{r \rightarrow \infty}\partial_t\int_0^\infty P_t^\lambda (x,y)u(y,r)dy=0,\quad x,t\in (0,\infty ).
\end{equation}

We note that the arguments that we will use to prove (\ref{3.9}) also allow us to obtain (\ref{M1}).
 
\begin{proof}[Proof of (\ref{3.9})] Since, for every $x,t\in (0,\infty )$, $\int_0^\infty P_t^\lambda (x,y)y^\lambda dy=x^\lambda$  (\cite[p. 84]{MS}) we can write
\begin{align}\label{3.9*}
\partial _t\int_0^\infty P_t^\lambda (x,y)u(y,r)dy&=\partial_t\int_0^\infty P_t^\lambda (x,y)y^\lambda y^{-\lambda}u(y,r)dy\nonumber\\
&=\partial _t\int_0^\infty P_t^\lambda (x,y)y^\lambda [y^{-\lambda} u(y,r )-x^{-\lambda }u(x,r)]dy\nonumber\\
&=\partial _t\int_0^\infty P_t^\lambda (x,y)y^\lambda \int_x^y\partial_z[z^{-\lambda} u(z,r )]dzdy,\quad x,t,r \in (0,\infty ).
\end{align}
Moreover, we have that
$$
\left|\int_x^y\partial_z[z^{-\lambda} u(z,r )]dz\right|\leq \int_x^y|D_{\lambda ,z}u(z,r)|z^{-\lambda}dz\leq C|y^{1-\lambda}-x^{1-\lambda}|\sup_{z\in I_{x,y}}|D_{\lambda ,z}u(z,r)|,\quad x,y,r\in (0,\infty ).
$$
Here, $I_{x,y}=[\min\{x,y\},\max\{x,y\}]$, $x,y\in (0,\infty )$.

Since $u$ is $\lambda$-harmonic in $(0,\infty )\times (0,\infty )$, we get
$$
\left(\partial_t^2-D_{\lambda ,x}D_{\lambda ,x}^*\right)D_{\lambda , x}u(x,t)=D_{\lambda ,x}\left(\partial_t^2-D_{\lambda ,x}^*D_{\lambda ,x}\right)u(x,t)=0,\quad x,t\in (0,\infty ).
$$
Note that
$$
-D_\lambda D_\lambda ^*=x^\lambda Dx^{-2\lambda }Dx^\lambda=u''-\frac{(\lambda +1)\lambda }{x^2}u=B_{\lambda +1}.
$$
Then, $D_{\lambda ,x}u$ is $(\lambda +1)$-harmonic in $(0,\infty )\times (0,\infty )$. Moreover, $x^{-\lambda -1}D_{\lambda ,x}u=\frac{1}{x}\partial _x(x^{-\lambda }u)$ is regular in $\mathbb{R}\times (0,\infty )$ and even in the $x$-variable. By proceeding as in the beginning of Section \ref{sec: (ii)(i)} after Lemma \ref{equiv} we can see that $(D_{\lambda ,x}u)^2$ is subharmonic in $\mathbb{R}\times (0,\infty )$.

Let $x,t\in (0,\infty )$. The subharmonicity of $(D_{\lambda ,x}u)^2$ allows us to write
\begin{align}\label{3.10*}
\sup_{z\in I_{x,y}}|D_{\lambda, z}u(z,r)|&\leq C\sup_{z\in I_{x,y}}\left(\frac{1}{r^2}\int _{B((z,r),r/4)}|D_{\lambda ,a}u(a,b)|^2dadb \right)^{1/2}\nonumber\\
&\leq C\left(\frac{1}{r^2}\int_{\frac{3r }{4}}^{\frac{5r}{4}}\int_{x-\frac{5r}{4}}^{x+\frac{5r}{4}}|D_{\lambda ,a}u(a,b)|^2dadb \right)^{1/2}\nonumber\\
&\leq \frac{C}{r}\|b|D_{\lambda ,a}u(a,b)|^2dadb\|_{\mathcal{C}}^{1/2},\quad |x-y|\leq r, \;y,r\in (0, \infty).
\end{align}

Also, we have that (see \cite[Lemma 3.2]{NoRo})
\begin{equation}\label{3.10}
|y^{1-\lambda }-x^{1-\lambda}|\leq C|x-y|\frac{\min \{x,y\}^{2-\lambda}}{xy},\quad y\in (0,\infty ).
\end{equation}

Then, by using \eqref{3.11*} we obtain
\begin{align}\label{3.11}
\left|\int_{0,|x-y|\leq r}^\infty \partial _tP_t^\lambda (x,y) y^\lambda \int_x^y\partial_z[z^{-\lambda }u(z,r)]dzdy\right| &\nonumber\\
&\hspace{-6cm}\leq C\frac{\|b|D_{\lambda ,a}u(a,b)|^2dadb\|_{\mathcal{C}}^{1/2}}{r}\int_{0,|x-y|\leq r}^\infty \frac{x^{\lambda -1}y^{2\lambda -1}|x-y|\min \{x,y\}^{2-\lambda}}{((x-y)^2+t^2)^{\lambda +1}}dy\nonumber\\
&\hspace{-6cm}\leq \frac{C}{r}\left[\int_{\max\{0,x-r\}}^x\frac{x^{\lambda -1}y^{\lambda +1}|x-y|}{((x-y)^2+t^2)^{\lambda +1}}dy+\int_x^{x+r }\frac{xy^{2\lambda -1}|x-y|}{((x-y)^2+t^2)^{\lambda +1}}dy\right]\nonumber\\
&\hspace{-6cm}\leq \frac{C}{r}\left[\frac{x^{\lambda -1}}{t^{2\lambda +1}}\int_0^x y^{\lambda +1}dy+\frac{x}{t^{2\lambda +1}}\int_x^{2x}y^{2\lambda -1}dy+x\int_{2x}^\infty\frac{y^{2\lambda}}{(y+t)^{2\lambda +2}}dy\right]\nonumber\\
&\hspace{-6cm}\leq \frac{C}{r}\left[\Big(\frac{x}{t}\Big)^{2\lambda +1}+\frac{x}{x+t}\right]\leq \frac{C}{r},\quad r\in (0,\infty ),
\end{align}
being $C$ depending on $x$ and $t$ but not on $r$.

We now make the following decomposition
\begin{align*}
y^{-\lambda}u(y,r)-x^{-\lambda }u(x,r)&=y^{-\lambda }[u(y,r )-u(y,|x-y|)]+y^{-\lambda} u(y,|x-y|)\\
&\hspace{0.5cm}-x^{-\lambda} u(x,|x-y|)+x^{-\lambda} [u(x,|x-y|)-u(x,r)]\\
&=-y^{-\lambda}\int_r^{|x-y|} \partial _su(y,s)ds+x^{-\lambda} \int_r ^{|x-y|}\partial _su(x,s)ds\\
&\hspace{0.5cm}+\int_x^yz^{-\lambda }D_{\lambda, z}u(z,|x-y|)dz,\quad y,r\in (0,\infty ).
\end{align*}
We get
\begin{align*}
|y^{-\lambda}u(y,r)-x^{-\lambda }u(x,r)|&\leq C\left[y^{-\lambda }\int_r ^{|x-y|}|\partial _su(y,s)|ds+x^{-\lambda} \int_r ^{|x-y|}|\partial _su(x,s)|ds\right.\\
&\hspace{0.5cm}\left.+|y^{1-\lambda}-x^{1-\lambda}|\sup_{z\in I_{x,y}}|D_{\lambda, z}u(z,|x-y|)|\right],\quad |x-y|>r, \;y,r\in (0,\infty ).
\end{align*}

From (\ref{3.10}), as in (\ref{3.10*}), it follows that
\begin{align*}
|y^{1-\lambda}-x^{1-\lambda}|\sup_{z\in I_{x,y}}|D_{\lambda, z}u(z,|x-y|)|&\leq C|x-y|\frac{\min \{x,y\}^{2-\lambda}}{xy}\sup_{z\in I_{x,y}}|D_{\lambda, z}u(z,|x-y|)|\\
& \leq C\frac{\min \{x,y\}^{2-\lambda}}{xy}\|b|D_{\lambda ,a}u(a,b)|^2dadb\|_\mathcal{C}^{1/2},\quad y\in (0,\infty ).
\end{align*}
By (\ref{3.2}) we get
\begin{align*}
y^{-\lambda}\int_r^{|x-y|}|\partial _su(y,s)|ds+ x^{-\lambda}\int_r^{|x-y|}|\partial _su(x,s)|ds&\leq C(x^{-\lambda}+y^{-\lambda})\|b|\partial _bu(a,b)|^2dadb\|_\mathcal{C}^{1/2}\int_r ^{|x-y|}\frac{ds}{s}\\
&\hspace{-4cm}\leq C\min\{x,y\}^{-\lambda}\log\frac{|x-y|}{r}\|b|\partial _bu(a,b)|^2dadb\|_\mathcal{C}^{1/2},\quad |x-y|>r,\;y,r\in (0,\infty ).
\end{align*}

Then, by (\ref{3.11*}),
\begin{align*}
\left|\int_{0,|x-y|>r}^\infty \partial_tP_t^\lambda (x,y)y^\lambda [y^{-\lambda}u(y,r)-x^{-\lambda }u(x,r)]dy\right|&\\
&\hspace{-7cm} \leq Cx^\lambda \int_{0,|x-y|>r}^\infty \frac{y^{2\lambda}}{((x-y)^2+t^2)^{\lambda +1}} \left(\min\{x,y\}^{-\lambda }\log\frac{|x-y|}{r}+
\frac{\min \{x,y\}^{2-\lambda}}{xy}\right)dy,\quad r\in (0,\infty ).
\end{align*}
We analyze each term separately. We have that, for every $r>x$, 
\begin{align*}
x^\lambda \int_{0,|x-y|>r}^\infty \frac{y^{2\lambda}\min \{x,y\}^{2-\lambda}}{((x-y)^2+t^2)^{\lambda +1}}\frac{dy}{xy}&\leq \frac{x}{r}\int_{x+r}^\infty \frac{y^{2\lambda-1}}{(y-x+t)^{2\lambda +1}}dy\leq C\frac{x}{r}\int_{2x}^\infty \frac{dy}{(y-x)^2}\leq \frac{C}{r},
\end{align*}
and
\begin{align*}
 x^\lambda \int_{0,|x-y|>r}^\infty \frac{y^{2\lambda}\min \{x,y\}^{-\lambda}}{((x-y)^2+t^2)^{\lambda +1}}\log\frac{|x-y|}{r}dy&\leq \int_{x+r}^\infty \frac{y^{2\lambda}}{(y-x+t)^{2\lambda +2}}\log\frac{y-x}{r}dy\\
&\hspace{-4cm}\leq C\int_{x+r}^\infty \frac{1}{(y-x)^2}\log\frac{y-x}{r}dy\leq \frac{C}{r}\int_1^\infty \frac{\log u}{u^2}du\leq \frac{C}{r}.
 \end{align*}
Here the constant $C$ can depend on $x$ and $t$, but not on $r$. 

We conclude that 
\begin{equation}\label{3.12}
\left|\int_{0,|x-y|>r}^\infty \partial _tP_t^\lambda (x,y)y^\lambda [y^{-\lambda}u(y,r)-x^{-\lambda }u(x,r)]dy\right|\leq\frac{C}{r},\quad r>x.
\end{equation}
By combining (\ref{3.9*}), (\ref{3.11}) and (\ref{3.12}) we deduce that (\ref{3.9}) holds.
\end{proof}

\begin{proof}[Proof of (\ref{3.8})] According to (\ref{3.3}) we have that
$$
\partial _tu(x,t+r)=\partial _ru(x,t+r)=P_t^\lambda \Big[\partial_r u(\cdot ,r )\Big](x),\quad x,t,r \in (0,\infty ).
$$
Since the differentiation under the integral sign is justified by the properties of the function $u$, we obtain
$$
\partial _r[u(x,t+r)-P_t^\lambda (u(\cdot ,r ))(x)]=0,\quad x,t,r \in (0,\infty ),
$$
and
$$
\partial_r\left[\partial_tu(x,t+r)-\partial_tP_t^\lambda (u(\cdot ,r ))(x)\right]=0,\quad x,t,r \in (0,\infty ).
$$
From (\ref{3.2}) and (\ref{3.9}) it follows that
$$
\lim_{r\rightarrow \infty }\left[\partial _tu(x,t+r)-\partial _tP_t^\lambda (u(\cdot ,r ))(x)\right]=0, \quad x,t\in (0,\infty ).
$$
Then,
$$
\partial_t[u(x,t+r)-P_t^\lambda (u(\cdot ,r ))(x)]=0,\quad x,t, r\in (0,\infty ).
$$
Also, (\ref{M1}) and Lemma \ref{Lem3.1} lead to
$$
\lim_{t\rightarrow 0^+}(u(x,t+r)-P_t^\lambda (u(\cdot ,r ))(x))=0,\quad x,r \in (0,\infty).
$$
We conclude that
$$
u(x,t+r)=P_t^\lambda (u(\cdot ,r ))(x),\quad x,t,r \in (0,\infty ),
$$
and (\ref{3.8}) is proved.
\end{proof}

For every $k\in \mathbb{N}$, we define
$$
u_k(x,t)=u\Big(x,t+\frac{1}{k}\Big),\quad x\in (0,\infty ),\;t\in [0,\infty ).
$$
We now establish that there exists $C>0$ such that
\begin{equation}\label{3.13}
\sup_I\frac{1}{|I|}\int_0^{|I|}\int_It|\partial _tu_k(x,t)|^2dxdt\leq C\|t|\partial _tu(x,t)|^2dxdt\|_\mathcal{C},
\end{equation}
where the supremum is taken over all bounded intervals $I\subset (0,\infty )$.

\begin{proof}[Proof of (\ref{3.13})] Let $k\in \mathbb{N}$ and let $I$ be a bounded interval in $(0,\infty )$. Suppose that $|I|\geq 1/k$. We obtain
\begin{align}\label{3.14}
\frac{1}{|I|}\int_0^{|I|}\int_It|\partial _tu_k(x,t)|^2dxdt&\leq \frac{1}{|I|}\int_0^{|I|}\int_I\Big(t+\frac{1}{k}\Big)\left|(\partial _tu)\Big(x, t+\frac{1}{k}\Big)\right|^2dxdt\nonumber\\
&\leq \frac{1}{|I|}\int_0^{2|I|}\int_{\widehat{I}}s|\partial _su(x,s)|^2dxds\leq 2\|s|\partial _su(x,s)|^2dxds\|_\mathcal{C},
\end{align}
where $\widehat{I}=(a,2b-a)$ when $I=(a,b)$ with $0\leq a<b<\infty$.

Assume now that $|I|<1/k$. According to (\ref{3.2}) we deduce that
$$
\left|\partial_t u\Big(x,t+\frac{1}{k}\Big)\right|\leq \frac{C}{t+1/k}\|s|\partial _su(x,s)|^2dxds\|_\mathcal{C}^{1/2},\quad x,t\in (0,\infty ).
$$
Then
\begin{align}\label{3.15}
\frac{1}{|I|}\int_0^{|I|}\int_It\left|\partial _tu\Big(x,t+\frac{1}{k}\Big)\right|^2dxdt&\leq \frac{C}{|I|}\|s|\partial _su(x,s)|^2dxds\|_\mathcal{C}\int_0^{|I|}\int_I\frac{t}{(t+1/k)^2}dxdt\nonumber\\
&\leq C\|s|\partial _su(x,s)|^2dxds\|_\mathcal{C}k^2\int_0^{|I|}tdt\leq C\|s|\partial _su(x,s)|^2dxds\|_\mathcal{C}.
\end{align}
Putting together (\ref{3.14}) and (\ref{3.15})  we prove (\ref{3.13}).
\end{proof}

We define, for every $k\in \mathbb{N}$, $f_k(x)=u_k(x,0)$, $x\in (0,\infty )$. By (\ref{3.8}), (\ref{M1}), (\ref{3.13}) and Lemma \ref{equiv} we obtain that, for every $k\in \mathbb{N}$, $f_k\in BMO_{\rm o}(\mathbb{R})$ and
\begin{equation}\label{3.16}
\|f_k\|_{BMO_{\rm o}(\mathbb{R})}\leq C\|s|\partial _su(x,s)|^2dxds\|_\mathcal{C}^{1/2}.
\end{equation}
Hardy spaces associated with Bessel operators have been studied in \cite{BDT} and \cite{Dz}. A function $f\in L^1(0,\infty )$ is in the Hardy space $H_{\rm o}^1(\mathbb{R})$ provided that
$$
\sup_{t>0}|P_t^\nu (f)|\in L^1(0,\infty ),
$$
for some (equivalently, for every) $\nu >1$. For every $\nu >1$, we define
$$
\|f\|_{H_\nu ^1}:=\|\sup_{t>0}|P_t^\nu (f)|\|_{L^1(0,\infty )},\quad f\in H^1_{\rm o}(\mathbb{R}).
$$
For each $\nu ,\mu >1$, the norms $\|\cdot\|_{H_\nu ^1}$ and $\|\cdot\|_{H_\mu ^1}$ are equivalent on $H_{\rm o}^1(\mathbb{R})$. The space $H_{\rm o}^1(\mathbb{R})$ endowed with the norm $\|\cdot\|_{H_\nu ^1}$ ($\nu >1$) is a Banach space. The dual space of $H_{\rm o}^1(\mathbb{R})$ is $BMO_{\rm o}(\mathbb{R})$ (\cite[Theorem 1]{DF}).

To finish the proof the following results will be useful.

\begin{lema}\label{Lem3.5}
Let $\nu >0$. For every $x,t\in (0,\infty )$, $P_t^\nu (x,\cdot)\in H_{\rm o}^1(\mathbb{R})$.
\end{lema}
\begin{proof}
Let $x,t\in (0,\infty )$. From the semigroup property it follows that
$$
P_s^\nu [P_t^\nu (x,\cdot )](z)=\int_0^\infty P_s^\nu (z,y)P_t^\nu (x,y)dy=P_{t+s}^\nu (x,z),\quad z,s\in (0,\infty ).
$$
According to \cite[p. 86, (b)]{MS} we have that
\begin{align*}
|P_s^\nu [P_t^\nu (x,\cdot )](z)|&\leq C\frac{(t+s)(xz)^\nu }{((t+s)^2+ (x-z)^2)^{\nu +1}}\leq C\frac{(xz)^\nu }{(t+s+|x-z|)^{2\nu +1}}\\
&\leq C\frac{(xz)^\nu }{(t+|x-z|)^{2\nu +1}}, \quad  z,s\in (0,\infty ).
\end{align*}
Then,
\begin{align*}
\int_0^\infty \sup_{s>0}|P_s^\nu [P_t^\nu (x,\cdot )](z)|dz&\leq C\left(\int_0^{2x}+\int_{2x}^\infty \right)\frac{(xz)^\nu }{(t+|x-z|)^{2\nu +1}}dz\\
&\hspace{-4cm}\leq C\left(\int_0^{2x}\frac{(xz)^\nu }{t^{2\nu +1}}dz+\int_{2x}^\infty \frac{(xz)^\nu }{(t+z)^{2\nu +1}}dz\right)\leq C\left(\Big(\frac{x}{t}\Big)^{2\nu +1}+\Big(\frac{x}{t}\Big)^\nu \right).
\end{align*}
Thus, we prove that $P_t^\nu (x,\cdot )\in H_{\rm o}^1(\mathbb{R})$.
\end{proof}

\begin{lema}\label{Lem3.6}
Assume that $g\in BMO_{\rm o}(\mathbb{R})$ and $G\in H_{\rm o}^1(\mathbb{R})$ satisfying that $gG\in L^1(0,\infty )$. Then,
\begin{equation}\label{gG}
\langle g,G\rangle_{BMO_{\rm o}(\mathbb{R}), H_{\rm o}^1(\mathbb{R})}=\int_0^\infty g(x)G(x)dx
\end{equation}
\end{lema}
\begin{proof}
According to the atomic characterization of $H_{\rm o}^1(\mathbb{R})$ (\cite[Theorem 1.10]{BDT}) we can find a sequence of measurable functions of compact support $(G_j)_{j\in \mathbb{N}}$ such that, for every $j\in \mathbb{N}$, $G_j$ is a linear combination of $H_{\rm o}^1(\mathbb{R})$-atoms, and $G_j\longrightarrow G$, as $j\rightarrow \infty$, in $H_{\rm o}^1(\mathbb{R})$. Then,
$$
\langle g,G\rangle_{BMO_{\rm o}(\mathbb{R}), H_{\rm o}^1(\mathbb{R})}=\lim_{j\rightarrow \infty }\langle g,G_j\rangle_{BMO_{\rm o}(\mathbb{R}), H_{\rm o}^1(\mathbb{R})}=\lim_{j\rightarrow \infty}\int_0^\infty g(x)G_j(x)dx.
$$

On the other hand, since $gG\in L^1(0,\infty )$ and $gG_j\in L^1(0,\infty )$, $j\in \mathbb{N}$, by \cite[p. 25]{BCR}, we have that
$$
\left|\int_0^\infty g(x)(G(x)-G_j(x))dx\right|\leq \|g\|_{BMO_{\rm o}(\mathbb{R})}\|G-G_j\|_{H_{\rm o}^1(\mathbb{R})},\quad j\in \mathbb{N}.
$$
By letting $j\rightarrow \infty$, we conclude (\ref{gG}).
\end{proof}

By using Banach-Alaoglu theorem and by taking into account (\ref{3.16}) there exists $f\in BMO_{\rm o}(\mathbb{R})$ and a strictly increasing $\phi: \mathbb{N}\longrightarrow \mathbb{N}$ such that $f_{\phi (k)}\longrightarrow f$, as $k\rightarrow \infty$, in the weak star topology of $BMO_{\rm o}(\mathbb{R})$, that is, for every $g\in H_{\rm o}^1(\mathbb{R})$,
\begin{equation}\label{3.17}
\langle f_{\phi (k)},g\rangle_{BMO_{\rm o}(\mathbb{R}), H_{\rm o}^1(\mathbb{R})}\longrightarrow \langle f,g\rangle _{BMO_{\rm o}(\mathbb{R}), H_{\rm o}^1(\mathbb{R})},\quad \mbox{ as }k\rightarrow \infty .
\end{equation}
Moreover,
$$
\|f\|_{BMO_{\rm o}(\mathbb{R})}\leq C\|s|\partial _su(x,s)|^2dxds\|_\mathcal{C}^{1/2}.
$$

By using (\ref{3.17}) and Lemma \ref{Lem3.5} we obtain, for  every $x,t\in (0,\infty )$,
$$
\langle f_{\phi (k)},P_t^\lambda (x,\cdot)\rangle_{BMO_{\rm o}(\mathbb{R}), H_{\rm o}^1(\mathbb{R})}\longrightarrow \langle f,P_t^\lambda  (x,\cdot )\rangle _{BMO_{\rm o}(\mathbb{R}), H_{\rm o}^1(\mathbb{R})},\quad \mbox{ as }k\rightarrow \infty .
$$

Since $BMO_{\rm o}(\mathbb{R})\subset BMO(\mathbb{R})$, by \cite[p. 86, (b)]{MS} and \cite[p. 141]{Ste2}, for every $g\in BMO_{\rm o}(\mathbb{R})$, we get
\begin{align*}
\int_0^\infty |g(y)||P_t^\lambda (x,y)|dy&\leq Ct\int_0^\infty \frac{|g(y)|}{1+y^2}dy\sup_{y\in (0,\infty )}\frac{1+y^2}{t^2+(x-y)^2}\\
&\leq Ct\left(\frac{1}{t^2}+\sup_{y\in (0,\infty )}\frac{y^2}{t^2+(x-y)^2}\right)\leq Ct\left(\frac{1+x^2}{t^2}+1\right),\quad x,t\in (0,\infty ).
\end{align*}
For every $x,t\in (0,\infty )$, Lemma \ref{Lem3.6} leads to
$$
\int_0^\infty f_{\phi (k)}(y)P_t^\lambda (x,y)dy\longrightarrow \int_0^\infty f(y)P_t^\lambda (x,y)dy,\quad \mbox{ as }k\rightarrow \infty .
$$
By (\ref{3.8}) we conclude that
$$
u(x,t)=\int_0^\infty f(y)P_t^\lambda (x,y)dy, \quad x,t\in (0,\infty ).
$$
Thus the proof is finished.

\def\cprime{$'$}

\end{document}